\pdfoutput=1 
\documentclass[reqno,12pt]{amsart}

\usepackage[dvipsnames]{xcolor}
\usepackage{microtype} 
\usepackage{amssymb,amsmath,amsthm,amsbsy,
 graphicx,
float, 
}
\usepackage{caption}
\usepackage{tikz}
\usetikzlibrary{decorations.pathreplacing,decorations.markings}
\tikzset{->-/.style={decoration={
  markings,
  mark=at position #1 with {\arrow{stealth}}},postaction={decorate}}}

\usepackage[margin=1in,letterpaper,portrait]{geometry}

\usepackage[shortlabels]{enumitem}

\usepackage[colorlinks=true, pdfstartview=FitV, linkcolor=purple, citecolor=blue,
filecolor=violet, urlcolor=violet]{hyperref}

\usepackage[all,cmtip]{xy}

\theoremstyle{definition} 
\newtheorem{theorem}{Theorem}[section] 
\newtheorem*{thm}{Theorem}
\newtheorem{lemma}[theorem]{Lemma}

\newtheorem{proposition}[theorem]{Proposition}
\newtheorem{corollary}[theorem]{Corollary}
\newtheorem{notation}[theorem]{Notation} 

\newtheorem{definition}[theorem]{Definition}
\newtheorem{example}[theorem]{Example}
\newtheorem{remark}[theorem]{Remark}
\numberwithin{equation}{section}
\counterwithout{figure}{section}

\usepackage[utf8]{inputenc}  
\usepackage[T1]{fontenc}

\newcommand{\MM}{\mathcal{M}}
\DeclareMathOperator{\cc}{CC}
\DeclareMathOperator{\gr}{Gr}
\DeclareMathOperator{\quid}{\mathcal Q}
\DeclareMathOperator{\tr}{tr}

\DeclareMathOperator{\punctureP}{\sf{P}}
\DeclareMathOperator{\punctureQ}{\sf{Q}}

\newcommand{\Quiver}{Q} 
\newcommand{\Subquiver}{Q'} 
\newcommand{\AnnulusQuiver}{Q'}

\newcommand{\rep}{{\rm rep\,}}
\newcommand{\mmod}{{\rm mod\,}}

\newcommand{\ag}{a}

\title{Infinite friezes of affine type D}

\author{Karin Baur, L\'ea\ Bittmann, Emily Gunawan, Gordana Todorov, Emine Y{\i}ld{\i}r{\i}m}

\address{University of Leeds, Ruhr-Universit\"at Bochum}
\address{Universit\'e de Strasbourg}
\address{University of Massachusetts Lowell}
\address{Northeastern University}
\address{University of Leeds}

\date{\today}

\begin{document}

\begin{abstract}
In this article, we study infinite friezes arising from cluster categories of affine type $D$ and determine the growth coefficients for these friezes. We prove that for each affine type $D$, the friezes given by the tubes all have the same growth behavior.
\end{abstract}

\maketitle

\setcounter{tocdepth}{2} 
\tableofcontents

\bigskip 

\section*{Acknowledgements} 
KB and EY are supported by the Royal Society Wolfson Award RSWF/R1/180004. KB is supported by the EPSRC programme grant EP/W007509/1. LB was partially supported by the European Research Council (ERC) under the European Union's Horizon 2020 research and innovation programme under grant agreement No 948885 and by the Royal Society University Research Fellowship. The work on this article began during the WINART3 Workshop: Women in Noncommutative Algebra and Representation Theory at the Banff International Research Station. The authors thank the BIRS and all the organizers. The authors would like to thank the Isaac Newton Institute for Mathematical Sciences, where some part of this work has been done, for support and hospitality during the programme “Cluster algebras and representation theory” supported by EPSRC grant no. EP/R014604/1.

\section{Introduction}

We are interested in frieze patterns of integers (first studied by Conway and Coxeter in the 1970s) which arise from tubes in a cluster category of affine type $D$. 
In general, tubes in a cluster category give rise to infinite friezes. In affine type $A$, there are exactly two non-homogeneous tubes, and it was shown in \cite{BFPT19} that infinite friezes corresponding to these tubes have the same growth coefficient. In general, this is not the case --- examples with tubes of different growth behavior arise from triangulations of a sphere with three boundary components or from Grassmannian cluster categories. Note that these examples are not of finite or affine type.
However, we will show that in affine type $D$ (where there are exactly three non-homogeneous tubes), their growth behavior is uniform.

\begin{thm}[Theorem~\ref{theo:main}]
For each affine type $D$, all infinite friezes have the same growth coefficient. 
\end{thm}

Friezes reside at the nexus of various mathematical structures, ranging from cluster variables in cluster algebras to representations of certain path algebras. A cluster algebra~\cite{FZ02} is a commutative algebra with distinguished generators called cluster variables. A surface cluster algebra~\cite{FST08} can be defined from a tagged triangulation; an affine type $A$ cluster algebra comes from a triangulation of an annulus (a sphere with two boundary components), and an affine type $D$ cluster algebra comes from a triangulation of a twice-punctured disk. The cluster variables correspond to (tagged) arcs in the corresponding surface. 
On the other hand, cluster categories~\cite{BMRRT06} were introduced to categorify the combinatorics of cluster algebras. The construction 
uses representations of path algebras. Throughout the paper we will be using the terms ``modules'' and ``representations'' 
interchangeably, depending on the context. Objects in a cluster category correspond to cluster variables under the cluster character map (also called Caldero--Chapoton map), or CC-map for short.

Friezes can be obtained from cluster categories by specializing the cluster character map. The fact is that the specialized CC-map counts submodules with multiplicities (arguing as in~\cite[Section 3.4]{BCJKT}),  
so we will focus 
on computing the numbers of submodules of the 
tubes. 
In order to do that we use the fact that the cluster category of affine type $D$ 
arises from a triangulation of a twice-punctured disk and that  objects of the cluster category of affine type $D$ correspond to (tagged) arcs in this surface~\cite{CL12, BQ15, QZ17, AP21}. 

The methods we use depend on the tubes: for the tube of rank $n-2$ in affine type $D_n$, we lift the setup to 
an affine type $A$
where the cluster category corresponds to a triangulation of an annulus (a sphere with two boundary components). We use the method of rank matrices 
from~\cite{Kan22}
to compute the number of submodules.

The two rank $2$ tubes in affine type $D$ arise from arcs ending at the punctures. 
For these two tubes, we use results from \cite{MSW13} about cluster variables corresponding to the quasi-simple objects at the mouth of these tubes. 

Finally, we complete our result by also considering the homogeneous tubes. Here, we use the fact that the quasi-simple object is given by a non-contractible closed curve in the annulus parallel to the boundary. 
The cluster algebra element
corresponding to this closed curve 
is the cluster character of the band module at the mouth of a homogeneous tube (i.e. the band module of quasi length $1$).

\section{Tagged triangulations and infinite friezes}

\subsection{Tagged triangulations}

In this section, we recall the notions of tagged arcs and triangulations, following~\cite[Section 2]{FST08}. Let $S$ be a connected, oriented Riemann surface with boundary. 
We also fix a finite set $\MM$ of marked points on the boundary and in the interior of $S$. The marked points in the interior are called \emph{punctures}. 
The pair $(S,\MM)$ is called a {\em marked surface}. In this article, we always assume that $(S,\MM)$ is a disk with two punctures and the boundary contains at least two marked points. 

\begin{definition}\label{def:arc}
An {\em arc} $\gamma$ in $(S,\MM)$ is a curve in $S$ with endpoints in $\MM$, considered up to isotopy fixing its endpoints. We require that 
\begin{itemize}
    \item[(i)] $\gamma$ does not cross  
    itself (however, its endpoints may coincide); 
    \item[(ii)] the interior of $\gamma$ is disjoint from the boundary of $S$; 
    \item[(iii)] $\gamma$ does not cut out an unpunctured monogon or  
    an unpunctured digon.
\end{itemize}
\end{definition}
Two arcs are called {\em compatible} if there exist representatives in their isotopy classes which do not cross. 

A {\em triangulation} of $(S,\MM)$ is a maximal collection of distinct, pairwise compatible arcs in $(S,\MM)$. Any triangulation of $(S,\MM)$ divides $S$ up into a collection of triangles. 
Note that since we work with punctured surfaces, some of the triangles may be self-folded. We call the arc inside a self-folded triangle its {\em radius} and the arc that encircles the puncture in the self-folded triangle its {\em loop}. 
An example of a triangulation of a disk with two punctures and four marked points on the boundary is on the left hand side in Figure~\ref{fig:tags-notags}; it contains two self-folded triangles.

A way to avoid self-folded triangles is to use \emph{tagged arcs} as in~\cite[Section 7]{FST08}: every loop arc of a self-folded triangle gets replaced by a tagged arc parallel to the radius, with a notch  
at the puncture. 

Let us formally define a tagged arc as follows. First, let $\gamma$ be an arc in $(S,\MM)$. This arc $\gamma$ has two ends. We mark each of its ends either as unnotched (plain) 
or as notched. This is called a \emph{tagging}. 
A {\em tagged arc} is an arc $\gamma$ with a choice of a tagging of its ends such that any endpoint at the boundary is unnotched. 
Also, if $\gamma$ is a loop which is not part of a self-folded triangle, both ends are tagged the same way. 
Two tagged arcs are shown on the right in Figure~\ref{fig:tags-notags}. 

Two different tagged arcs $\gamma_1,\gamma_2$ are said to be {\em compatible} if their underlying arcs are compatible and if the following holds: 
If the untagged versions of $\gamma_1$ and $\gamma_2$ coincide then they must have the same tagging at least at one of their common endpoints, if they do not coincide but have a common endpoint, then they are tagged the same way at that endpoint.

In certain situations it will be more convenient to use self-folded triangles (with loops and radii).  

\begin{definition}\label{def:tagged-triangulation}
    Let $(S,\MM)$ be a marked surface. 
    A {\em tagged triangulation} of $(S,\MM)$ is a maximal collection of distinct pairwise compatible tagged arcs in $(S,\MM)$. 
\end{definition}

\begin{example}
An example of a tagged triangulation is on the right hand side in Figure~\ref{fig:tags-notags}.  
\begin{figure}[ht]
\includegraphics[scale=.7]{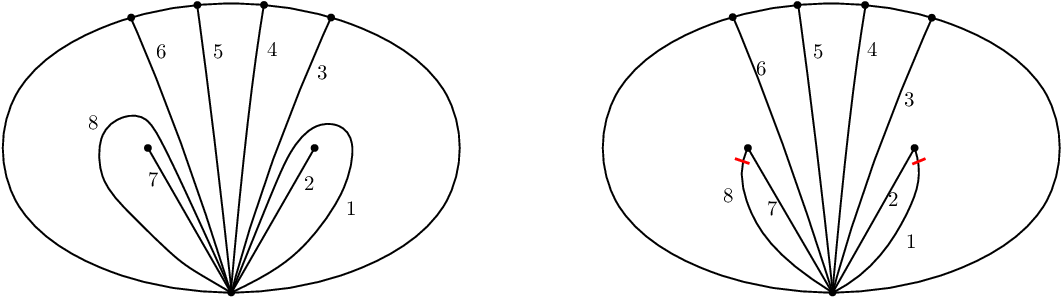}
\caption{Twice-punctured disk with triangulation and tagged triangulation}
\label{fig:tags-notags}
\end{figure} 
\end{example}

\begin{remark}\label{rem:number-of-arcs}
Any triangulation of a twice-punctured disk with $n-2$ marked points on the boundary has  $n+1$ arcs. To see this, one can use \cite[Proposition 2.10]{FST08}. 
\end{remark}

\subsection{Quivers from tagged triangulations and from arcs}\label{sec:quiver}

Given a triangulation $T$, we can write a quiver $\Quiver_T$ corresponding to $T$, called the \emph{quiver of $T$}. We draw a vertex of $\Quiver_T$ for each arc of $T$, and the arrows of $\Quiver_T$ are defined as follows:

\begin{itemize} 
\item Step 1 (draw an arrow for each angle in a non self-folded triangle):
For each of the three distinct ``ideal angles'' of every non self-folded $\triangle$ in the triangulation $T$, we do the following: Let $\angle$ denote this angle, and let   $\alpha_i,\alpha_j$ denote the distinct sides of $\triangle$ which form  $\angle$. 

If $\alpha_j$ is clockwise from $\alpha_i$ when going around $\angle$, then we draw an arrow from $i$ to $j$; and if $\alpha_i$ is clockwise from $\alpha_j$ when going around $\angle$, we draw an arrow from $j$ to $i$. Let $\alpha_\angle$ denote the arrow of $\Quiver_T$ corresponding to this angle.
\item 
Step 2 (draw arrows from or to the radius of a self-folded triangle): 
For every self-folded triangle $\triangle$, we do the following:
Let $\alpha_i$ be the loop and $\alpha_j$ the radius of $\triangle$.  
We draw an arrow $j \to k$ (from the radius to a third arc $\alpha_k$) whenever there is already an arrow $i \to k$ (from the loop to arc $\alpha_k$); similarly, we draw an arrow $k \to j$ whenever there is already an arrow $k \to i$.  
\item Step 3: We remove all two-cycles which may have been created during Step 1. 
\end{itemize}

\[
\includegraphics[scale=.8]{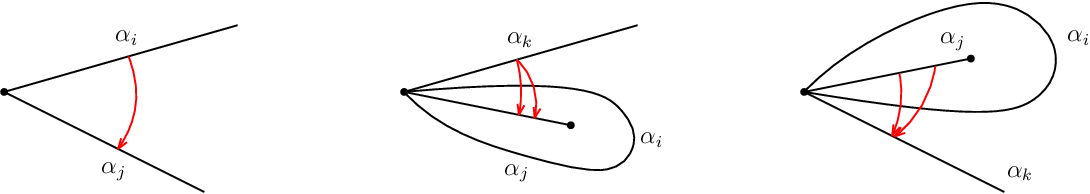}
\]

\hskip.2cm

As an example, the quiver $\Quiver_T$ for triangulation $T$ in Figure~\ref{fig:tags-notags} is the following.
\[
\xymatrix@R=.6pt@C=6pt{7\ar[rd]  &&&&& 2 \\
& 6\ar[r] & 5\ar[r] & 4\ar[r] & 3\ar[ru]\ar[rd] &\\
8\ar[ru] &&&&& 1 }
\]
As we see, the quiver $\Quiver_T$ in this case is an affine type $D$ quiver: $\widetilde{D}_7$. 

Given a triangulation $T$ and an arc $\gamma$ in the surface which is not in $T$, we can write a quiver $\Quiver_\gamma$ corresponding to $\gamma$, called \emph{the quiver of $\gamma$}. 
First, we choose representatives in the isotopy classes so that $\gamma$ crosses each arc of $T$ a minimum number times. The vertices $i$ of $\Quiver_\gamma$ correspond to the arcs $\alpha_i$ of $T$ which are crossed by $\gamma$. 
The arrows of $\Quiver_\gamma$ are defined as follows: 
\begin{itemize}
\item 
Step 1 (take an arrow for each angle of a non self-folded triangle crossed by $\gamma$):
For each angle $\angle$ crossed by $\gamma$ such that $\angle$ belongs to  a non self-folded triangle, we do the following.

\begin{enumerate}[(a)]
\item If $\Quiver_T$ contains the arrow $\alpha_\angle$, then take $\alpha_\angle$ to also be an arrow in $\Quiver_\gamma$. 
\item If $\Quiver_T$ contains no arrow $\alpha_\angle$, we do nothing in this step. (The reason we might be in this situation is that a two-cycle between the vertices was removed in Step 3 of the construction of the quiver $\Quiver_T$.) 
\end{enumerate}
\item Step 2 (take arrows from or to the radius of a self-folded triangle): 
For each pair $i, j$ of distinct vertices of $\Quiver_\gamma$ such that 
 $\alpha_i$ and $\alpha_j$ are the arcs of a self-folded triangle $\alpha$, say $\alpha_i$ the loop and $\alpha_j$ the radius, this means that $\gamma$ crosses $\triangle$ at least once in such a way that $\gamma$ crosses the loop $\alpha_i$, then the radius $\alpha_j$, then the loop again. In this case, as in Step 2 of the previous definition, we take the arrow $j \to k$ (from the radius to a third arc $\alpha_k$) to be an arrow in $\Quiver_\gamma$ whenever there is already an arrow $i \to k$ (from the loop to arc $\alpha_k$) in $\Quiver_\gamma$; similarly, we take the arrow $k \to j$ whenever there is already an arrow $k \to i$ in $\Quiver_\gamma$.  
\end{itemize}

See the right hand of Figure~\ref{fig:regions} for two quivers of arcs.

We will not associate vertices to the boundary segments of the triangulation. Such vertices are called frozen vertices in the literature and we will not consider quivers with frozen vertices. 
We note that any triangulation of a twice-punctured disk will give rise to a quiver which is mutation equivalent to an affine type $D$ quiver~\cite[Example 6.10]
{FST08}.

\subsection{Friezes} 

In this subsection, 
we recall the definition and some important statements about friezes. In the next subsection (Subsection~\ref{sec:growth-coeff})
we talk about an invariant for friezes called \emph{growth coefficient}.

\begin{definition} A \emph{frieze} (or \emph{frieze pattern}) is a staggered array of possibly infinitely many rows of integers, 
starting with a row of $0$s and a row of $1$s, followed by either a finite number of 
rows of positive integers, and end with a row of $1$s followed by a row of $0$s, or by infinitely many rows 
of positive integers, satisfying the so-called \emph{diamond rule} (or \emph{SL$_2$-rule}): $ad-bc=1$ is satisfied for any four entries 
$$
\xymatrix@R=.6pt@C=1pt{
 & b \\
a && d \\
 & c}
$$
\end{definition} 

If the frieze contains only finitely many rows it is called {\em finite (or closed) frieze}. 
The rows of $0$s and of $1$s are called the {\em trivial rows}. 
By the diamond rule, any three neighboring entries determine 
the fourth. In particular, starting from the top, the first non-trivial row determines  
the whole frieze. 

Friezes, or sometimes called Conway-Coxeter friezes, were introduced by Coxeter in~\cite{C71} and were studied by Conway and Coxeter in the celebrated work~\cite{CC73} where they showed that any finite frieze has a translational symmetry. 
If the translational 
symmetry has order $m$ (i.e. if $m$ is the {\em smallest} positive integer such that the frieze is invariant under horizontal translation by $m$), then any $m$ consecutive entries in the first non-trivial row can 
be used to 
determine the frieze. 
This motivated the definition of a {\em quiddity sequence} of a frieze: 
as 
any $n$-tuple formed by $n$ consecutive entries in the first non-trivial row which can be used to generate the entire frieze (``quiddity'' for 
essence of the frieze). Note that $n$ can be a multiple of $m$ (since $m$ is minimal).

In particular, Conway and Coxeter established a link between finite friezes and triangulations of convex polygons through the following construction. By an $n$-gon we mean a convex polygon with $n$ vertices, labeled counterclockwise by $1,2,\dots, n$ around the boundary.

\begin{theorem}[\cite{CC73}]\label{thm:cc}
The following is a bijection between friezes with $n-3$ non-trivial rows and triangulations of $n$-gons: Let $T$ be a triangulation of the $n$-gon and let $a_i$ be the number of triangles of $T$ which are incident to vertex $i$, for 
$i=1,\dots,n$. Then $(a_1,a_2,\dots,a_n)$ is the quiddity sequence of a frieze with 
$n-3$ non-trivial rows, and every such frieze arises from a triangulated $n$-gon.
\end{theorem}

Various generalizations of friezes have been studied. The above finite friezes can be associated to cluster algebras in Dynkin type $A$~\cite[Section 5]{CC06}. 
Friezes for other (extended) Dynkin types have been studied for example 
in~\cite{BM2009},~\cite{ARS10},~\cite{AD11},~\cite{KS11},~\cite{GM22} and~\cite{BCJKT}.

In this paper, we work with infinite periodic friezes --- friezes which contain infinitely many rows of positive integers and which have a translational symmetry. 
We refer to them as infinite friezes for short. 
Infinite friezes also arise from triangulations of surfaces, for example from punctured disks~\cite{Tsch15} and pairs of pants~\cite{CFGT21}.
More generally, every infinite (periodic) frieze can be obtained from a triangulation of an annulus: 

\begin{theorem}[Theorems 3.7 and 4.6, \cite{BPT16}]\label{thm:infinite-annulus}
(1) Every triangulation of an annulus with $n$ marked points on the outer boundary and $m$ marked points on the inner boundary gives two infinite friezes by taking as quiddity sequences the number of triangles locally incident with vertices on the outer boundary: $\quid_1=(a_1,\dots, a_n)$ and the number of triangles locally incident with vertices on the inner boundary: 
$\quid_2=(b_1,\dots, b_m)$. 

(2) Let $\quid=(a_1,\dots, a_n)$
be the quiddity sequence of an infinite frieze. Then there exists a triangulation of an annulus with $n$ marked points on the outer boundary such that the $a_i$ are the number of triangles at the vertices of the outer boundary.     
\end{theorem}

In part (1) of Theorem~\ref{thm:infinite-annulus} we wrote ``locally incident'' to emphasize that we have to consider the triangles seen in a small neighborhood of any vertex. See Section~\ref{sec:affine-friezes} and in particular Remark~\ref{rem:dots} for more details.

When writing out a frieze, we often omit the row of $0$s at the top. 
We label the rows starting with the constant row of $1$s (as row $0$), and the quiddity sequence is in row $1$.

\begin{example}\label{ex:annulus-frieze}
Here is an example of an infinite 
frieze with quiddity sequence $(4,2,2)$ of length $n=3$. We omit the trivial row of $0$s.

\[
\begin{tikzpicture}[scale=.6, inner sep=4pt]

\node at (-11,5) {Row $0$};
\node at (-11,4) {Row $1$};
\node at (-11,3) {Row $2$};
\node at (-11,2) {Row $3$};
\node at (-11,1) {Row $4$};
\node at (-11,0) {Row $5$};

\node (10) at (-7,5) {$\ldots$};
\node (11) at (-5,5) {1};
\node (12) at (-3,5) {1};
\node (12) at (-1,5) {1};
\node (12) at (1,5) {1};
\node (14) at (3,5) {1};
\node (15) at (5,5) {1};
\node (16) at (7,5) {$\ldots$};

\node (ai+) at (-6,4) {$\ldots$};
\node (ai)  at (-4,4) {4};
\node (ai+) at (-2,4) {2};
\node (ai+) at (0,4)  {2};
\node (aj-) at (2,4)  {4};
\node (aj)  at (4,4)  {2};
\node (ai+) at (6,4)  {2};
\node (ai+) at (8,4)  {$\ldots$};

\node (10) at (-7,3) {$\ldots$};
\node (11) at (-5,3) {7};
\node (12) at (-3,3) {7};
\node (12) at (-1,3) {3};
\node (12) at (1,3)  {7};
\node (14) at (3,3)  {7};
\node (15) at (5,3)  {3};
\node (16) at (7,3)  {$\ldots$};

\node (17) at (-6,2) {$\ldots$};
\node (18) at (-4,2) {12};
\node (19) at (-2,2) {10};
\node (20) at (0,2) {10};
\node (21) at (2,2)  {12};
\node (22) at (4,2)  {10};
\node (23) at (6,2)  {10};
\node (24) at (8,2)  {$\ldots$};

\node (10) at (-7,1) {$\ldots$};
\node (11) at (-5,1) {17};
\node (12) at (-3,1) {17};
\node (12) at (-1,1) {33};
\node (12) at (1,1)  {17};
\node (14) at (3,1)  {17};
\node (15) at (5,1)  {33};
\node (16) at (7,1)  {$\ldots$};

\node (17) at (-6,0) {$\ldots$};
\node (18) at (-4,0) {24};
\node (19) at (-2,0) {56};
\node (20) at (0,0) {56};
\node (21) at (2,0)  {24};
\node (22) at (4,0)  {56};
\node (23) at (6,0)  {56};
\node (24) at (8,0)  {$\ldots$};

\node (ai)  at (-4,-1) {$\vdots$};
\node (ai+) at (-2,-1) {$\vdots$};
\node (ai+) at (0,-1)  {$\vdots$};
\node (aj-) at (2,-1)  {$\vdots$};
\node (aj)  at (4,-1)  {$\vdots$};


\end{tikzpicture}
\]
%
A triangulation for this frieze is shown in the following figure: the number of 
triangles meeting the vertices of the outer boundary are $\quid_1=(4,2,2)$. 
\[
\includegraphics[width=5cm]{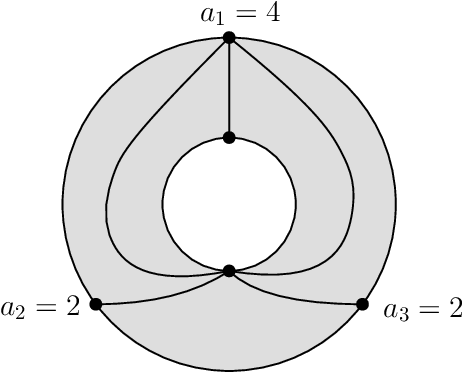}
\]
The inner boundary component yields the quiddity sequence $\quid_2=(2,5)$ of length $n=2$ for the following infinite frieze:

\[
\begin{tikzpicture}[scale=.6, inner sep=4pt]

\node at (-11,5) {Row $0$};
\node at (-11,4) {Row $1$};
\node at (-11,3) {Row $2$};
\node at (-11,2) {Row $3$};
\node at (-11,1) {Row $4$};

%
\node (10) at (-6,5) {$\ldots$};
\node (11) at (-4,5) {1};
\node (12) at (-2,5) {1};
\node (12) at (0,5) {1};
\node (12) at (2,5) {1};
\node (14) at (4,5) {1};
\node (16) at (6,5) {$\ldots$};

\node (ai+) at (-5,4) {$\ldots$};
\node (ai)  at (-3,4) {5};
\node (ai+) at (-1,4) {2};
\node (ai+) at (1,4)  {5};
\node (aj-) at (3,4)  {2};
\node (ai+) at (5,4)  {$\ldots$};

\node (10) at (-6,3) {$\ldots$};
\node (11) at (-4,3) {9};
\node (12) at (-2,3) {9};
\node (12) at (0,3) {9};
\node (12) at (2,3)  {9};
\node (14) at (4,3)  {9};
\node (16) at (6,3)  {$\ldots$};

\node (ai+) at (-5,2) {$\ldots$};
\node (ai)  at (-3,2) {16};
\node (ai+) at (-1,2) {40};
\node (ai+) at (1,2)  {16};
\node (aj-) at (3,2)  {40};
\node (ai+) at (5,4)  {$\ldots$};

\node (10) at (-6,1) {$\ldots$};
\node (11) at (-4,1) {71};
\node (12) at (-2,1) {71};
\node (12) at (0,1) {71};
\node (12) at (2,1)  {71};
\node (14) at (4,1)  {71};
\node (16) at (6,1)  {$\ldots$};

\node (ai)  at (-3,0) {$\vdots$};
\node (ai+) at (-1,0) {$\vdots$};
\node (ai+) at (1,0)  {$\vdots$};
\node (aj-) at (3,0)  {$\vdots$};


\end{tikzpicture}
\]
\end{example}

\subsection{Growth coefficients}\label{sec:growth-coeff}

It is proven in \cite[Section 2]{BFPT19} that in any $m$-periodic frieze, the difference between any entry in 
row $m$ 
and the entry right above it in row $m-2$ is constant. Therefore, the following notion is well-defined.

\begin{definition}\label{def:growth-coeff}
Let $\quid=(a_1,\dots, a_m)$ be a quiddity sequence of 
a frieze $F$. The \emph{growth coefficient} 
$s(\quid)$ 
of $F$ is the difference between any entry in row $m$ 
and the entry right above it in row $m-2$. 
If $m=1$, we consider the trivial row of $0$s as the row $-1$.
\end{definition}

We note that such constant differences occur repeatedly throughout the frieze in the following sense: 
For any $k>0$, the differences between any 
entry in row $km$ and the entry above it, in row $km-2$, are always constant. 
In fact, these constant differences are given as a Chebyshev polynomial in $s(\quid)$, by \cite[Proposition 2.10]{BFPT19}. 
Let $s_k(\quid)$ 
denote the constant 
difference between entries in row $km$ and $km-2$, for $k\ge 1$. (In particular, $s_1(\quid)=s(\quid)$.) Then if we set $s_0(\quid)=2$, we have 
\begin{equation}\label{eq:recurrence}
    s_{k+2}(\quid)=s_1(\quid)s_{k+1}(\quid)-s_k(\quid) \mbox{ for $k\ge 0$}.
\end{equation}
For example, $s_2(\quid)=s_1(\quid)s_1(\quid)-2=s(\quid)^2-2$. 

Furthermore, if the quiddity sequence has a translational symmetry (i.e. if there is a smaller translational period in the frieze, say by $m_0$, with $m=d m_0$), there are 
already constant differences in earlier rows, 
e.g. between the entries in row $m_0$ and the entries just above.

This is illustrated in Example~\ref{ex:period}.

\begin{example}\label{ex:period}
Consider the following infinite frieze: 
\[
\xymatrix@R1pt@C3pt{
 1 && 1 && 1 && 1 && 1 && 1 && 1 && 1 && 1  \\
  & 2 && 3 && 2 && 3 && 2 && 3 && 2 && 3 \\
 && 5  && 5  && 5  && 5  && 5  && 5  && 5 \\ 
 & && 8 && 12 && 8 && 12 && 8 && 12 \\
&& && 19 && 19 && 19 && 19 && 19 \\
 & && && 30 && 45 && 30 && 45 \\
 && && && 71 && 71 && 71 \\
 & && && && \vdots && \vdots
}
\]
This can be viewed as a frieze with 
quiddity sequence $\quid=(2,3)$. Since $n=2$, the growth coefficient is $s(\quid)=5-1=4$. If we view it as a frieze with quiddity sequence $\quid'=(2,3,2,3)$, then $n=4$ and the growth coefficient is $s(\quid')=19-5=14$. We note that $s(\quid')=s(\quid)^2-2$, cf. Remark~\ref{rem:repeat-q}\ref{rem:repeat-q:itm:i}: 
the growth 
coefficient $s(\quid')$ 
is equal to $s_2(\quid)$ arising from $s(\quid)$  
as in equation \eqref{eq:recurrence}. 
\end{example}

\begin{remark}\label{rem:repeat-q}
\begin{enumerate}[(i)]
\item 
\label{rem:repeat-q:itm:i}  
    If $k\cdot \quid$ denotes $k$ repetitions of the quiddity sequence $\quid$, then $s(k\cdot \quid)=s_k(\quid)$.
\item 
In~\cite[Theorem 3.4]{BFPT19}, the authors show that the two growth coefficients of the pair of friezes arising from the two boundary components of a triangulated annulus coincide. See Example~\ref{ex:annulus-frieze}: the growth coefficient of the frieze with quiddity sequence $(4,2,2)$ is $8$, since the difference between entries in 
    row $3$ 
    and the row above is $12-4=8=10-2$. 
    The second frieze has quiddity sequence $(5,2)$, and its  growth coefficient is $9-1=8$. 
\end{enumerate}
\end{remark}

We note that this phenomenon is very special. In general for triples of friezes various configurations of growth coefficients might occur. Consider the three friezes arising from a triangulation of a sphere with three boundary 
components (a pair of pants). The growth coefficients of these friezes can vary; see Example~\ref{ex:pair-pants} below. We expect that it is never the case that the three growth coefficients from a triangulation of a pair of pants are all equal. 
\begin{example}\label{ex:pair-pants}
    The following triangulation of a pair of pants gives three 1-periodic friezes with quiddity sequences $(4)$, $(5)$ and $(6)$, respectively. Their growth coefficients 
    are 
    $4-0=4,5-0=5$ and $6-0=6$ respectively, which are all 
    different.  
    \[ 
    \includegraphics[width=5cm]{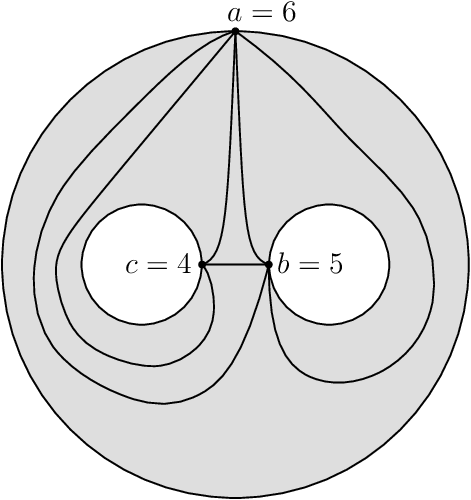}
    \]
\end{example}

However, there are cases of triples of friezes, coming from a different surface, with the same growth coefficient. 
In this paper, we prove that the growth coefficients of the triple of friezes arising in affine type $D$ (stemming from a disk with two punctures) are the same (Theorem \ref{theo:main}).

\section{Cluster and module categories of affine type $D$}~\label{clu-cat}

In this section, we recall the structure of the non-homogeneous tubes in the cluster category of affine type $D$ and then explain how we obtain three different friezes which we call \emph{affine type $D$ friezes}.

Let $k$ be an algebraically closed field. Consider an algebra $\Lambda$ of affine type $D$. We can assume that $\Lambda=kQ$ 
where the quiver $\Quiver$  
is any orientation of the following 
extended Dynkin diagram with $n+1$ vertices: 

\[
\includegraphics[width=7cm]{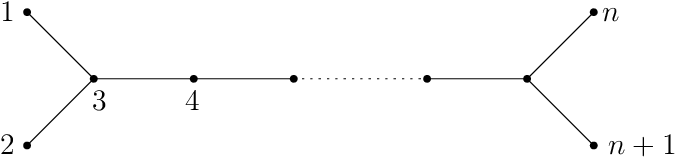}
\]

Recall that throughout the paper we will be using both terms: modules and representations; sometimes we need module-theoretic properties and sometimes it is convenient to construct representations. 
It is known that in our situation, we have an equivalence of categories between $\rep Q$ and the module category $\mmod kQ$, where $\rep Q$ is the category of finite dimensional $k$-representations of $\Quiver$ (with $k$ algebraically closed), and $\mmod kQ$ is the category of finitely generated modules over the associated path algebra $kQ$ (see for example~\cite[Chapter III, Cor.1.7]{ASS06}).

\subsection{Tubes in the cluster category of 
affine type $D$}\label{sec:three-tubes}

Module categories and cluster categories of affine type $D$ are of infinite type but still well-studied. These categories are determined by their indecomposable objects and the morphisms between them. Most of the information we need for these categories is encoded in the so-called Auslander--Reiten quiver (AR-quiver for short) of the category which has vertices for isomorphism classes of indecomposable objects and arrows for irreducible morphisms. Algebras of affine type $D$ are tame and the AR-quivers of their module categories have three non-homogeneous tubes: one of rank $n-2$, which we call 
$\mathcal{T}_1$, and two tubes of rank $2$, which we call $\mathcal{T}_2$ and $\mathcal{T}_3$ 
\cite{DR76}, see also~\cite[Section XIII, 2.2 Corollary(b)]{SS07}.

The same is true for the associated cluster category 
\[{\mathcal C}_Q=\mathcal{C}_{\Lambda}=D^b({\rm mod}\,\Lambda)/\tau^{-1}[1],\]
first defined in \cite{BMRRT06}. Here $\tau$ is the Auslander--Reiten translate on $D^b({\rm mod}\,\Lambda)$, i.e. it is equal to $\tau$ on nonprojective objects in $\rm mod\,\Lambda$ and for $P$ projective, $\tau P=I[-1]$ where $I$ is the injective envelope of the top of $P$, and $[-]$ is the shift functor in $D^b({\rm mod}\,\Lambda)$
%
We keep the 
notations $\mathcal{T}_1,\mathcal{T}_2$ and $\mathcal{T}_3$ for the three tubes in $D^b({\rm mod}\,\Lambda)$.

We recall the definition of a \emph{standard category} as a category which is equivalent to the mesh category of its AR-quiver. When an algebra is of infinite representation type, neither the module category nor the cluster category is standard. However when considering only the full subcategory of a tube, there is a difference: a tube in a module category is standard, but the same tube in a cluster category is not standard since there are homomorphisms in the infinite radical. While our geometric model is for the tubes in cluster categories, we only consider compositions of irreducible maps (these are not in the infinite radical). With this condition on homomorphisms, these tube categories are standard, 
see~\cite{BM2012}. 

\subsection{Arcs in  
a twice-punctured disk and quasi-simple modules}\label{sec:arc-model}

Geometric models have been fruitfully used in the study of cluster categories. One of the important ingredients of the proof of our main result is the surface model for affine type $D$. 
The surface which gives cluster algebras and cluster categories of affine type $D_n$ is a twice-punctured disk, with $n-2$ marked points on the boundary, see e.g.~\cite[Example 6.10]{FST08} and~\cite[Section 6]{QZ17}. 
In particular, we use the characterization of indecomposable objects of the three non-homogeneous tubes in terms of arcs in a surface, see also~\cite{BM2012} and~\cite{AP21}. 

Recall that a module at the mouth of a tube is called a \emph{quasi-simple} module. 
The quasi-simple modules give rise to the quiddity sequences of the friezes if we take their CC-map and specialize to integers (see Section~\ref{sec:affine-friezes} below). 
As in the affine type $A$ setting (see Theorem \ref{thm:infinite-annulus} above),
these integers correspond  to the number of triangles (locally) incident with the marked points on the boundary, see Figure~\ref{fig:regions}. 
We recall in this subsection the geometric descriptions of the arcs for the quasi-simple modules, see Proposition~\ref{prop:indec-tubes} and Lemma~\ref{lm:arcs at the mouth of small tubes}.

Let $(S,\MM)$ denote a twice-punctured disk with $n-2$ marked points on the boundary, labeled counterclockwise along the boundary as $\{1,2,\dots, n-2\}$, and let $\punctureP$ and $\punctureQ$ denote the two punctures in the disk. 
The indecomposable objects of the cluster category (or the module category) of affine type $D$ can be described by generalized arcs in $(S,\MM)$: 
A {\em generalized arc} is an isotopy class of curves satisfying only conditions (ii) and (iii) of Definition~\ref{def:arc}: a generalized arc is allowed to cross itself. 
Indecomposable modules correspond to generalized arcs, 
see e.g.~\cite{BZ11} in the unpunctured case and~\cite{BQ15, AP21} for punctured surfaces. 

\begin{example}
    We distinguish four types of (generalized) arcs in a twice-punctured disk $(S,\MM)$:
    \begin{enumerate}[(i)]
        \item 
        arcs with both endpoints on the boundary and which are homotopic to a concatenation of boundary segments as in the first picture of Figure~\ref{fig:generalized-arcs},
        \item 
        arcs with both endpoints on the boundary but which are not homotopic to a concatenation of boundary segments (second picture),
        \item 
        arcs with one end on the boundary and one end at a puncture (third picture),
        \item 
        arcs with both endpoints at punctures (fourth picture). 
    \end{enumerate}
The arcs of type (i) are called {\em peripheral}. 
\end{example} 

\begin{figure}[ht]
\centering
\includegraphics[scale=.8]{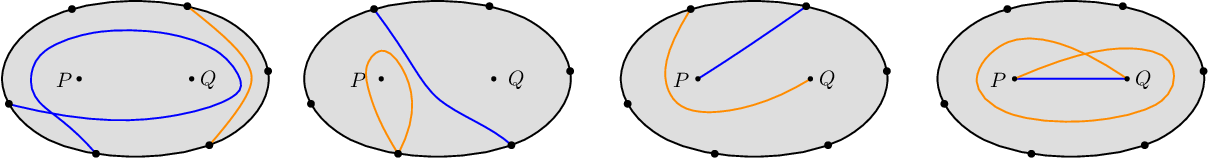}
\caption{Four different types of (generalized) arcs in a twice-punctured disk.}
\label{fig:generalized-arcs}
\end{figure}

The description of the Auslander--Reiten translate in terms of arcs depends on the endpoints of the arcs. 

We will concentrate on the effect of the Auslander--Reiten translate on (generalized) arcs of type (i) and (iv) since these arcs give rise to indecomposable objects in the three non-homogeneous tubes as we will see below (Corollary~\ref{cor:periodic-arcs}). 

We first introduce notation for peripheral arcs: 
Such an arc $\gamma$ starts at a vertex $v_s\in\{1,\dots, n-2\}$, follows the boundary counterclockwise, homotopic to a concatenation of at least $2$ boundary segments and then ends at another vertex of the boundary, $v_t\in\{1,\dots, n-2\}$. The notation $\gamma_{v_s,v_t}^l$ is for the generalized arc which indicates the starting and ending points and the number $l$ of full turns around the boundary the arc does: 
It may go around the two punctures more than once and cross itself. 
Important examples of peripheral arcs are: the arc $\gamma_{v,v+2}^0$ (with $v\le n-4$) which connects $v$ with $v+2$ and forms a triangle with the two boundary segments $(v,v+1)$ and $(v+1,v+2)$ adjacent to $v+1$. 
Similarly, $\gamma_{n-3,1}^0$ forms a triangle with the two boundary segments incident with $n-2$, and $\gamma_{n-2,2}^0$ forms a triangle with the two boundary segments incident with $1$. 
For every $v\in \{1,\dots, n-2\}$, the arc  $\gamma_{v,v}^1$ is a loop starting and ending at $v$, going around the two punctures once

The following two lemmas describe the effect of $\tau$ on (generalized) arcs in the surface:

\begin{lemma}
[{\cite[Proposition 1.3]{BZ11}}]\label{lm:tau-bigtube}

Let $\gamma$ be a generalized arc in $(S,\MM)$ with both endpoints on the boundary. Let $M=M(\gamma)$ be the corresponding indecomposable object in the cluster category. Then the arc of the Auslander--Reiten translate $\tau(M(\gamma))$ is obtained by moving both endpoints of $\gamma$ clockwise to the next marked points on the boundary. 

In particular, if 
$\gamma=\gamma_{v_s,v_t}^l$ is a peripheral generalized arc,  the arc of the Auslander--Reiten translate $\tau(M(\gamma))$ is 
 $\gamma_{v_s-1,v_t-1}^l$ (decreasing the labels of endpoints by $1$ modulo $n-2$). 
\end{lemma}

\begin{lemma}[{\cite[Section 3.2]{BQ15}}]\label{lm:tau-tags} 
\begin{enumerate}
    \item Let $\gamma$ be an arc with endpoints at the two punctures. Let $M(\gamma)$ be the corresponding indecomposable object. Then the arc of $\tau(M(\gamma))$  is obtained by changing tags at both ends. 
    \item Let $\delta$ be an arc with one endpoint $i$ on the boundary and one endpoint at a puncture. Let $M(\delta)$ be the corresponding indecomposable object. Then the 
arc of $\tau(M(\delta))$
is obtained by changing the tag at the puncture and by moving the endpoint on the boundary to $i-1$. 
\end{enumerate}
\end{lemma}

\begin{corollary}\label{cor:periodic-arcs}
Let $M=M(\gamma)$ be an indecomposable object. If $M$ is $\tau$-periodic then $\gamma$ cannot be of type (ii) or (iii). 
\end{corollary}

\begin{proof}
    Applying $\tau$ repeatedly to (generalized) arcs of type (ii) or (iii) introduces higher and higher winding around the punctures and hence such modules cannot be $\tau$-periodic.
\end{proof}

Recall that we write $\mathcal T_1$ for the tube of rank $n-2$ and $\mathcal T_2$, $\mathcal T_3$ for the two rank 2 tubes of the cluster category of affine type $D$. Let Ind$(\mathcal T_i)$ denote the set of indecomposable objects of $\mathcal T_i$ (up to isomorphism). 
For the purpose of friezes we only need to consider compositions of irreducible maps between such objects. 

\begin{proposition}
\label{prop:indec-tubes}
Let $(S,\MM)$ be the twice-punctured disk with $n-2$ marked points on the boundary. Then there is a bijection: 
\[
\mbox{Ind}(\mathcal{T}_1) 
\longleftrightarrow 
\{\mbox{peripheral generalized arcs } \gamma_{v_s,v_t}^l \},
\]
i.e. with the set $\gamma_{v_s,v_t}^l$ with 
$0\le v_s,v_t\le n-2,\ l\ge 0$,  and where $v_t\notin \{v_s, v_s+1\}$ if $l=0$. Furthermore: 
\begin{enumerate}
\item 
The indecomposable quasi-simple objects of the tube $\mathcal T_1$ correspond to the arcs 
$\gamma_{v,v+2}^0$ (reducing modulo $n-2$), for $v=1,\dots, n-2$; 
\item 
The irreducible maps in $\mathcal T_1$ correspond to the lengthening and shortening of generalized arcs (reducing endpoints modulo $n-2$): 

\begin{eqnarray}
\mbox{(right) lengthening} 
& 
\gamma_{v_s,v_t}^l 
& \to \left\{ \begin{array}{ll}
\gamma_{v_s,v_t+1}^l &  \mbox{ if $v_s \ne v_t+1$} 
\\ 
& \\
\gamma_{v_s,v_t+1}^{l+1} & \mbox{ else.}
\end{array}
\right. \label{eq:lengthen}
\\ 
\mbox{(left) shortening} 
& \gamma_{v_s,v_t}^l & \to 
\left\{ 
\begin{array}{ll}
\gamma_{v_{s+1},v_t}^l &  \mbox{ if $l>0$ and $v_t\ne v_s$} \\ 
& \\
\gamma_{v_{s+1},v_s}^{l-1} & \mbox{ $l>0$ and $v_t=v_s$} \\ 
 & \\
\gamma_{v_{s+1},v_t}^0 & \mbox{ if $l=0$ and $v_t\ne v_s+2$.}\\
\end{array}
\right.\label{eq:shorten}
\end{eqnarray}
\end{enumerate}
\end{proposition} 
\begin{proof}
We show that the set of generalized arcs in the statement defines a translation quiver which is a tube of rank $n-2$, hence isomorphic to the Auslander--Reiten quiver of $\mathcal{T}_1$.

We define a quiver $\Gamma_{n-2}$ whose vertices are the generalized arcs 
$\gamma_{u_s,u_t}^l$, 
as illustrated here (where for $l=0$, $v_t\notin\{v_s,v_s+1,v_s+2\}$):  
\[
\xymatrix@R=8pt@C=8pt{
 & \gamma_{v_s,v_t+1}^{\lambda(l)}\ar[dr] & & \\ 
\gamma_{v_s,v_t}^l\ar[ur] \ar[dr]& & \gamma_{v_s+1,v_t+1}^l\ar@{.>}[ll]_{\tau} &    \\
 & \gamma_{v_s+1,v_t}^{\sigma(l)}\ar[ur] & 
}
\]
with $\lambda(l)\in\{l,l+1\}$, see \eqref{eq:lengthen} and $\sigma(l)\in \{l-1,l\}$, see \eqref{eq:shorten}.
We always reduce endpoints modulo $n-2$. 
The arrows of $\Gamma_{n-2}$ are defined as 
in \eqref{eq:lengthen} and \eqref{eq:shorten} above. 
Furthermore, we have a translation map on the vertices (as in Lemma~\ref{lm:tau-bigtube}) which we also denote by $\tau$: 
\[
\tau:\gamma_{v_s,v_t}^l \mapsto \gamma_{v_s-1,v_t-1}^l.
\]

\vskip .3cm

\noindent
So locally, we have a configuration as above (unless $l=0$ and $v_t=v_s+2$) or as illustrated here (when $l=0$ and $v_t=v_s+2$): 
\[
\xymatrix@R=8pt@C=8pt{
 & \gamma_{v,v+3}^0\ar[rd]  \\
\gamma_{v,v+2}^0\ar[ur] & &
\gamma_{v+1,v+3}^0
\ar@{.>}[ll]_{\tau}
}
\]
(if $n-2=2$, the arc in the middle is $\gamma_{v,v+1}^1$). 

Then one can check that $(\Gamma_{n-2},\tau)$ is a stable translation quiver in the sense of Riedtmann, \cite{R80}. Furthermore, it is isomorphic to the Auslander--Reiten quiver of $\mathcal{T}_1$. 
For details, we refer to~\cite[Sections 2.3, 2.4]{BM2012}. 
This proves parts (1) and (2).  
\end{proof}

Proposition~\ref{prop:indec-tubes} is illustrated in Figure~\ref{fig:AR-quiver}.

\begin{figure}[htbp!]
\includegraphics[width=115mm]{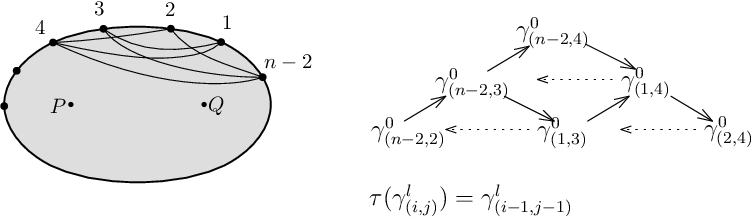}
\caption{Part of the AR-quiver corresponding of the tube of rank $n-2$.}\label{fig:AR-quiver}
\end{figure}

\begin{remark}\label{rem:n=4}
In case $n=4$, there are three tubes of rank two as $\mathcal T_1$ is also of rank two, with the quasi-simples in tube $\mathcal T_1$ given by the arcs $\gamma_{1,3}=\gamma_{1,3}^0$ and $\gamma_{2,4}=\gamma_{2,4}^0$ (Figure~\ref{fig:arcs-at-mouths}, right picture). 
\end{remark}

To describe the arcs for the quasi-simples in the rank two tubes $\mathcal{T}_2$ and $\mathcal{T}_3$, we can use the fact that the quasi-simples are the only rigid modules in these tubes. And rigid modules correspond to arcs in the sense of Definition~\ref{def:arc}, i.e. without self-crossings. 
We show in Lemma~\ref{lm:arcs at the mouth of small tubes} that the corresponding arcs are the four possible tagged arcs starting at one puncture and ending at the other puncture. We denote them by 
$\gamma$, 
$\gamma^{(\punctureP)}$, $\gamma^{(\punctureQ)}$
and 
$\gamma^{(\punctureP\punctureQ)}$ (see Figure~\ref{fig:arcs-at-mouths}, left two pictures). 
\begin{center}
\begin{figure}[htbp!]
\includegraphics[scale=.89]{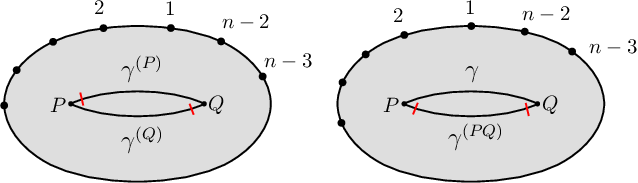}
\hskip 1cm\includegraphics[scale=.89]{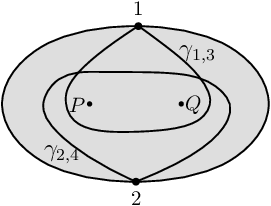}
\caption{Arcs corresponding to quasi-simples in rank $2$ tubes: the tubes $\mathcal{T}_2$ and $\mathcal{T}_3$ for arbitrary $n$ (left and center), the tube $\mathcal{T}_1$ for $n=4$ (right).}
\label{fig:arcs-at-mouths}
\end{figure}
\end{center}

\begin{lemma}\label{lm:arcs at the mouth of small tubes} 
The arcs corresponding to the quasi-simple modules of the tubes $\mathcal T_2$ and $\mathcal T_3$ are $\{\gamma^{(\punctureP)},\gamma^{(\punctureQ)},\gamma,\gamma^{(\punctureP\punctureQ)}\}$. 
Furthermore, we have 
\[
\begin{array}{lcl}
\tau(\gamma^{(\punctureP)})= \gamma^{(\punctureQ)} & \mbox{and}& \tau(\gamma^{(\punctureQ)})= \gamma^{(\punctureP)} \\ 
\tau(\gamma)= \gamma^{(\punctureP\punctureQ)} & \mbox{and}& \tau(\gamma^{(\punctureQ\punctureP)})= \gamma 
\end{array}
\]
\end{lemma}

\begin{proof} 
Arcs corresponding to modules in a rank 2 tube have to be $\tau$-periodic of period $2$. 
By Corollary~\ref{cor:periodic-arcs}, the only candidates for $2$-periodic arcs are $\gamma^{(\punctureP)}$, $\gamma^{(\punctureQ)}$, $\gamma$, $\gamma^{(\punctureP\punctureQ)}$ from Figure~\ref{fig:arcs-at-mouths} or, in case $n=4$, the two arcs, each starting and ending at the same boundary point and forming a loop around the two punctures, see Figure~\ref{fig:arcs-at-mouths}. The modules of the latter are  in ${\mathcal T}_1$.  These two arcs correspond to $\gamma_{1,3}^0$ and $\gamma_{2,4}^0$ in the notation of Proposition~\ref{prop:indec-tubes} and $\tau$ interchanges them. See also Remark~\ref{rem:n=4}.

By Lemma~\ref{lm:tau-tags}, $\tau$ interchanges $\gamma$ and  $\gamma^{(\punctureP\punctureQ)}$ and it interchanges  $\gamma^{(\punctureP)}$ and $\gamma^{(\punctureQ)}$.
\end{proof}

\subsection{Friezes of 
affine 
type $D$}\label{sec:affine-friezes}

This section is devoted to showing that the non-homogeneous tubes give rise to a triple of infinite friezes. 
Recall from~\cite{CC06, P08} that the CC-map associates to any indecomposable module $M$ the following rational function: 
\begin{equation}~\label{eq:CC}
\centering
    \cc(M)=\frac{1}{x_1^{d_1}x_2^{d_2}\cdots x_n^{d_n}}\sum_{e\in \mathbb{N}^{\Quiver_0}} \chi(\gr_e(M))\prod_{i\in \Quiver_0}x_i^{\sum_{j\to i}e_j+\sum_{i\to j}(d_j-e_j)}
\end{equation}
where

\begin{enumerate}[i.]
    \item 
    $(d_1,d_2,\cdots,d_n)=\underline{\rm dim}\,M$,
    \item 
    $\gr_e (M)$ is the quiver Grassmannian and $\chi$ the Euler--Poincar\'{e} characteristic.
\end{enumerate} 
In that sense we think of friezes as functions on the AR-quiver. 
For details, we refer to the last section of 
the lecture notes~\cite{Mor17}, survey \cite{Mor15} and \cite{ARS10, AD11}.

For any indecomposable object $M$, the value of the  specialized (generalized) CC-map is exactly the number of submodules of $M$ (taking multiplicities into account), arguing as in~\cite{BCJKT}. 
So, to get frieze patterns in affine type $D$, we will compute the number of submodules of the modules in the non-homogeneous tubes.

In these tubes, the numbers of submodules of quasi-simple modules (i.e. modules in row 1) give us the quiddity sequence of the associated frieze. All other entries are determined by the diamond rule. For $M$ a module, let us write $\rho(M)$ for the number of submodules of $M$.

In the remaining part of this section, we compute the numbers $\rho(M)$ for quasi-simple modules in the tube $\mathcal{T}_1$: let $\quid=(a_1,\dots, a_{n-2})$ be the associated quiddity sequence. 
By the above and using 
Proposition~\ref{prop:indec-tubes} (1), the entry $a_i$ is the number of submodules of the module corresponding to the arcs $\gamma_{i-1,i+1}$. We note that the coefficients $\chi(\gr_e(M))$ in~\eqref{eq:CC} are explicitly computed in~\cite{C11} for $M$ corresponding to these type of arcs.
 
\begin{remark}\label{rem:module-arc}
Here we describe the representation 
$M(\gamma)$ corresponding to a peripheral arc $\gamma=\gamma_{i-1,i+1}=\gamma_{i-1,i+1}^0$ (reducing endpoints modulo $n-2$). 
Let $r_i$ be the number of arcs of $T$ ending at $i$ (counting every appearing loop arc as one arc). 

If 
$\gamma$ is an arc of the triangulation then $r_i=0$. In this case, we consider $M(\gamma)$ to be a zero module and $\rho(M(\gamma))=1=r_i+1$. 

So assume now that $\gamma$ is not an arc of the triangulation. 
The quiver associated to $\gamma$ is a subquiver $\Quiver_{\gamma}$ of the quiver $\Quiver_T$ of the triangulation $T$. 
We distinguish three situations (see Figure~\ref{fig:arc-quiver}): 
\begin{enumerate}[(a)]
    \item \label{rem:module-arc:itm:a} there is no loop at $i$, 
    \item \label{rem:module-arc:itm:b} there is one loop at $i$,
    \item \label{rem:module-arc:itm:c} there are two loops at $i$. 
\end{enumerate}
\begin{figure}[htbp]
\centering
\captionsetup{justification=centering}
\includegraphics[scale=.9]{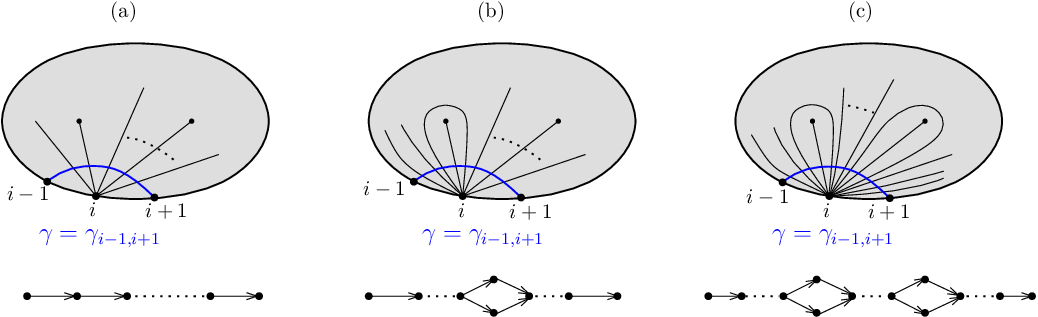}
\caption{Arcs $\gamma$ and associated quivers $\Quiver_\gamma$ for quasi-simple modules in $\mathcal T_1$.}
\label{fig:arc-quiver}
\end{figure}

Note that in all cases $\Quiver_\gamma$ has $r_i$ vertices.
In case \ref{rem:module-arc:itm:a}, $\Quiver_{\gamma}$ is a linear orientation of a quiver of type $A_{r_i}$. 
In cases \ref{rem:module-arc:itm:b} and \ref{rem:module-arc:itm:c} the figure shows one or two rhombi formed by the arrows, with two vertically aligned and two horizontally aligned vertices. These rhombi are not always present. 
In case \ref{rem:module-arc:itm:b}, the quiver may end with the two vertices vertically aligned (i.e. the quiver doesn't contain the vertices to the right of these two vertices), or it may start with the two vertices vertically aligned (i.e. the quiver doesn't contain the vertices to the left of these two vertices). 
Similarly, in case \ref{rem:module-arc:itm:c}, the quiver may end with the two vertically aligned vertices on the right or it may start at the two vertically aligned vertices on the left. Furthermore, in \ref{rem:module-arc:itm:c}, it may happen that the two rhombi share the vertex in the middle. 
See Figure~\ref{fig:regions} for two examples of the quivers $\Quiver_{\gamma}$ for cases \ref{rem:module-arc:itm:a} and \ref{rem:module-arc:itm:b}. 

In each of these three cases, the indecomposable representation $M(\gamma)$ is one-dimensional at every vertex of $\Quiver_{\gamma}$ and has identity maps on the arrows. 
This describes $M(\gamma)$ as a representation of $\Quiver_\gamma$ and also of $\Quiver_T$. 
\end{remark}

It is straightforward to prove the 
following result. 
\begin{lemma}\label{lm:quid-peripheral}
Let $\gamma_{i-1,i+1}=\gamma_{i-1,i+1}^0$ be a peripheral arc (reducing endpoints modulo $n-2$) and let $r_i$ be the number of arcs incident with $i$ as above. Then the number of submodules of the quasi-simple module
$M(\gamma_{i-1,i+1})$ is the following:
\[
\rho(M(\gamma_{i-1,i+1}))  = 
\left\{
\begin{array}{ll} r_i+1 & \mbox{in case (a)}, \\
r_i+2 & \mbox{in case (b)}, \\
r_i+3 & \mbox{in case (c)}. \end{array}
\right.
\]
\end{lemma}

\begin{proof}
Consider Figure~\ref{fig:arc-quiver}. 
Let $\Quiver_{\gamma}$ be the quiver of the peripheral arc $\gamma=\gamma_{i-1,i+1}^{0}$, as defined in Section \ref{sec:quiver}. In all three case, $\Quiver_{\gamma}$ is equioriented. 
Label its vertices from sink to source as follows
\[
\begin{array}{lc}
(a) & 
\xymatrix{
r_i\ar[r] & r_i-1\ar[r] & \dots & \dots\ar[r] & 2\ar[r] & 1
} \\ 
  \\
(b) & \xymatrix@R=.6pt@C=14pt{ & & &   s'\ar[rd] & & \\
r_i-1\ar[r] & \dots\ar[r] &  s+1\ar[ru]\ar[rd]  &  & s-1\ar[r] & \dots\ar[r] & 1 \\ 
& & &  s\ar[ru] &  & & &
} \\
(c) & \xymatrix@R=.6pt@C=14pt{ & & & t'\ar[rd] &&&&  s'\ar[rd] & & \\
r_i-2\ar[r] & \dots \ar[r] & t+1 \ar[ru]\ar[rd] &  &t-1\ar[r]& \dots\ar[r] &s+1\ar[ru]\ar[rd]  &  & s-1\ar[r] & \dots\ar[r] & 1 \\ 
& & & t\ar[ru] &&&& s\ar[ru] &  & & &
}
\end{array}
\]

In case $(a)$, the non-zero submodules correspond exactly to the connected subquivers supported at vertices $1, 2, \dots, s$, for $1 \leq s \leq r_i$, 
so there are $r_i$ non-zero submodules. In total, $\rho(M(\gamma))=r_i+1$. 

In cases (b) and (c), this is similar: the non-zero submodules correspond to 
one of the following: 
\begin{itemize}\item 
The submodule generated by a single vertex; there are $r_i$ such submodules since $\Quiver_\gamma$ has $r_i$ vertices. 
\item 
We also have the submodule generated by the vertices $s$ and $s'$ together.
\item In addition, for case (c), we have the submodule generated by the vertices $t$ and $t'$ together.
\end{itemize}
Taking all these submodules together with the zero module, we get
$\rho(M(\gamma))=r_i+2$ in case (b) and $\rho(M(\gamma))=r_i+3$ in case (c), as claimed.


\end{proof}

\begin{remark}\label{rem:dots} 
Determining the entries $a_i$ in the quiddity sequence amounts to counting regions near boundary vertices. 
In the situation of 
Figure~\ref{fig:arc-quiver}, in a small neighborhood of $i$, we see $r_i+1$ regions in the case (a) (where there are no loops at $i$), $r_i+2$ regions in case (b), and $r_i+3$ regions in case (c). 

So, by Lemma~\ref{lm:quid-peripheral}, in all three situations the number $\rho(M(\gamma_{i-1,i+1}))$ of submodules of $M(\gamma_{i-1,i+1})$ is equal to the number of regions incident with $i$ when restricting the triangulated surface to a small neighborhood of $i$. 
In Figure~\ref{fig:regions}, these neighborhoods are indicated by the blue curves near the vertices on the boundary, and the regions are indicated by the dots. 
\end{remark}

\begin{example}\label{ex:triangul-quidd}
Consider the triangulation in Figure~\ref{fig:regions}. We see that the quiddity sequence of the frieze from the tube $\mathcal{T}_1$ is $\quid=(7,1,4,2,2)$ (see Remark~\ref{rem:dots}). 
The figure also shows two peripheral arcs (middle picture) and the corresponding representations (on the right).

Figure~\ref{fig:AR-frieze} shows part of the AR-quiver of this tube and several rows of the frieze given by this triangulation. 
Note that as is customary, 
the AR-quiver is drawn going up the page and the frieze going down. 
We will see later (Lemma~\ref{lm:quid-small-tubes}) that the quiddity sequences of the friezes of the two rank two tubes are $\quid_2=(3,12)$ and $\quid_3=(3,12)$. 
\end{example}

\begin{figure}[htbp!]
\centering
\includegraphics[scale=.99]{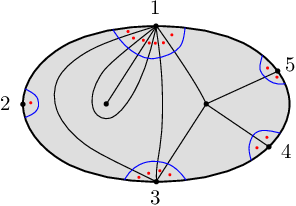}
\hskip 1cm\includegraphics[scale=.99]{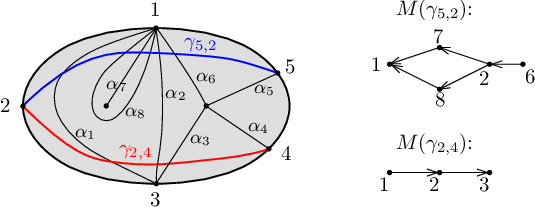}
\caption{Quiddity sequence $\quid=(7,1,4,2,2)$ from the 
boundary of a triangulated surface 
and some quasi-simple modules of the corresponding tube~$\mathcal T_1$.}
\label{fig:regions}
\includegraphics[scale=.64]{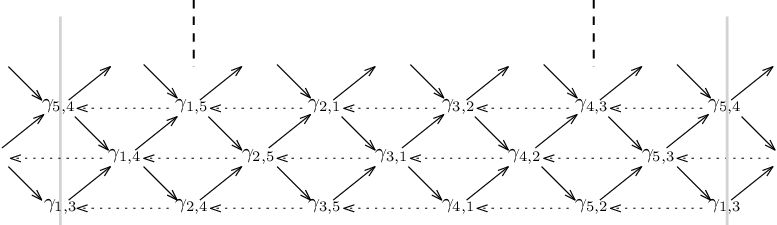}
\hskip.5cm
\includegraphics[scale=.64]{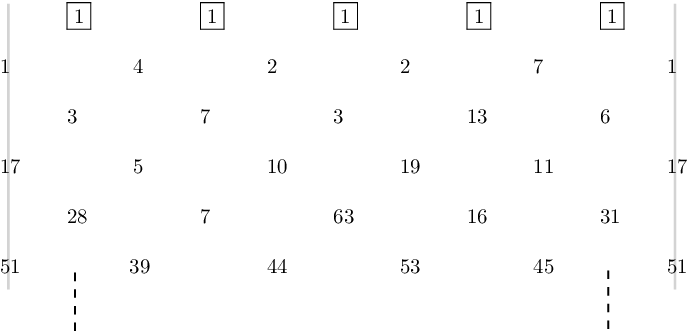}
\caption{Part of the AR-quiver of the tube $\mathcal T_1$ and of the corresponding frieze for the triangulation from Figure~\ref{fig:regions}. The row of $1$s (highlighted with boxes) has no counterpart in the AR-quiver on the left.}
\label{fig:AR-frieze}
\end{figure}

%
\section{Growth coefficient of the tube $\mathcal{T}_1$ of rank ($n-2$)}\label{sec:big-tube}

To compute the growth coefficient for the tube $\mathcal{T}_1$ of rank $n-2$, we  determine the values of 
rows $n-2$ and $n-4$ in the frieze associated to this tube 
(see Definition~\ref{def:growth-coeff}). 

Our strategy is to two-fold. First we translate the setup from affine type $D$ to affine type $A$: 
By Theorem~\ref{thm:infinite-annulus} (2), every infinite frieze can be obtained from a triangulation of an annulus. 
In a second step,  we use a result of~\cite{GMV19} to compute the growth coefficient directly from a closed loop in the annulus. 

We note that once we have translated to affine type $A$, the quivers of the representations corresponding to peripheral arcs we consider are of type $A$. 
This enables us to use the techniques of the theory of order ideals of fence posets.

We first recall the necessary background on order ideals (Section~\ref{sec:order-ideal}) then explain how we can go from a twice-punctured disk to an annulus (Section~\ref{sec:D-to-A}) and finally compute the growth coefficient of the tube $\mathcal{T}_1$ (Section~\ref{sec:outer-boundary}).

\subsection{Posets and order ideals}\label{sec:order-ideal}

We will explain here the correspondence between quivers of type $A$ and fence posets, and how submodules then correspond to order ideals. This allows us to use the works~\cite{Kan22} (see also subsequent work~\cite{KY22}) to compute the number of submodules of our type $A$ modules, i.e. the number of order ideals, via computation of $2\times2$ matrices.

We now recall the necessary notation and results.

A {\em fence poset} $P$ is a poset which 
comes from a quiver $\Quiver$ whose underlying graph is a type $A$ Dynkin diagram. The elements of $P$ are the vertices of $\Quiver$, and the covering relation of $P$
is given by the arrows of $\Quiver$: we have $i$ is covered by $j$  
if and only if there is an arrow from vertex $j$ to vertex $i$ 
in $\Quiver$.

Maximal chains of a fence poset are called \emph{segments}. A fence poset $P$ can therefore be represented by a vector $(|\alpha_1|,|\alpha_2|,\cdots,|\alpha_m|)$ 
where $|\alpha_i|$ is the number of arrows in the $i$-th segment $\alpha_i$ of the poset, reading the poset from left to right.

The segment $\alpha_i$ with vertices $\{v_1,\dots,v_s \}$, read from left to right, is a {\em direct segment} if $v_j>v_{j+1}$ for $1\le j<s$; it is an {\em inverse segment} if $v_j<v_{j+1}$ for $1\le j<s$.

An \emph{order ideal} $I$ in $P$ is a down-closed subset of $P$, i.e. 
for every $x\in I$, if  $y\in P$ and $y\leq_P x$, then 
$y$ is also in $I$.

\begin{example}\label{ex:fence}
In this example, we consider an orientation of a type $A_5$ Dynkin diagram as on the left of Figure~\ref{fig:poset-example}. 
The Hasse diagram of the poset $P$ corresponding to it is shown on the right.

\begin{figure}[htbp!]
\centering
\hspace{5mm}
{\xymatrix{1\ar[r] & 2\ar[r] & 3 & 4\ar[l] \ar[r] & 5}
}
\hspace{1cm}
\xymatrix{
 1 \ar[dr] &    &   & & \\
   & 2 \ar[dr]  &   & 4 \ar[dr]  \ar[dl]  & \\
   &                 & 3 & & 5}
   \hspace{5mm}
\caption{
A type $A_5$ quiver 
seen on the right as the Hasse diagram of a poset.}\label{fig:poset-example}
\end{figure}
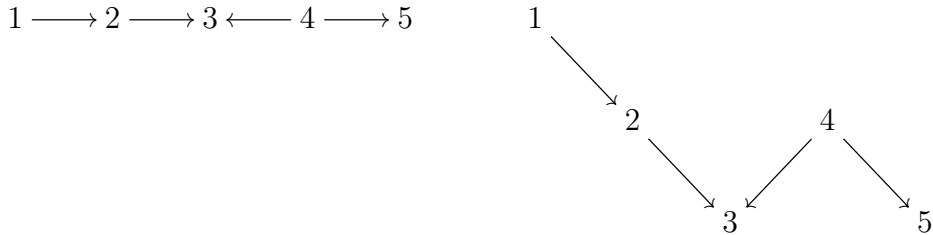

The poset $P$ has 11 order ideals: 
\[
\emptyset, 
\{3\}, \{5\}, \{2,3\}, \{3,5\}, \{1,2,3\}
\{2,3,5\}, \{3,4,5\}, \{1,2,3,5\}, \{2,3,4,5\}, \{1,2,3,4,5\}.
\]
The segments of $P$ are $1\to 2\to 3$, $3\leftarrow 4$, and $4\to 5$; 
the segments $1\to 2 \to 3$ and $4\to 5$ are direct, and $3 \leftarrow 4$ is inverse. 
The vector of this fence poset $P$ is $(2,1,1)$. 
\end{example}

For the rest of this subsection, let $\Quiver$ be a Dynkin quiver of type $A$ and $P$ its corresponding fence poset.

In Section \ref{sec:quiver}, we define the quiver $\Quiver_\gamma$ of an arc $\gamma$ to be a certain subquiver of the quiver $\Quiver_T$ of a triangulation. In similar spirit, we would like to define representations from subquivers.
\begin{definition}[From subquivers to modules]
\label{def:representation defined by a subquiver}
In general, given a subquiver $\Subquiver$ of $\Quiver$, let $M(\Subquiver)$ denote the representation of $\Quiver$ which has one copy of the field at each vertex of $\Quiver$ and the identity map for every arrow in $\Subquiver$; the rest of the vertices of $\Quiver$ are assigned the zero vector space, and the rest of the arrows the zero map.

When $\Subquiver$ is $\Quiver$ itself, $M(\Subquiver)=M(\Quiver)$ is an example of a {\em sincere} representation (of $\Quiver$) since it is non-zero at every vertex of $\Quiver$. In fact, it is the only indecomposable sincere representation, since $\Quiver$ is of type $A$.

When $\Subquiver$ is the empty subquiver, $M(\Subquiver)$ is the zero representation. 
\end{definition}

Note that the map from each non-empty, connected subquiver $\Subquiver$ of $\Quiver$ to the module $M(\Subquiver)$ gives a one-to-one correspondence between the non-empty connected subquivers of $\Quiver$ and indecomposable representations of $\Quiver$. 
Given an indecomposable module $M$, we can draw and identify the module $M$ with the corresponding subquiver $\Subquiver$ of $\Quiver$ which is equal to the support of $M$, i.e. we only draw the part of $\Quiver$ where $M$ is non-zero; we refer to this subquiver $\Subquiver$ as the \emph{quiver of $M$} and denote it by $\Quiver_M$. We will sometimes think of $\Quiver_M$ as an abstract directed graph. 

Furthermore, there is a one-to-one correspondence between order ideals of $P$ and submodules of $M(\Quiver)$ given by sending each order ideal $I$ of $P$ to the representation $M(FullQuiv(I))$, where $FullQuiv(I)$ is the full subquiver of $\Quiver$ whose vertices are $I$.  
So counting the number of submodules of $M(P)$ is the same as counting order ideals of $P$.

To count these order ideals, we will apply the method of Kantarc{\i} O{\u{g}}uz,~\cite{Kan22},  using a procedure with $2\times 2$ matrices arising recursively from smaller posets. 
We now describe her results using the language of quiver representations.

\begin{definition}
\label{def:rm} 
Let $\Quiver$ be a quiver of type $A_n$ with vertices $1$ through $n$, drawn so that the vertices are in increasing order from left to right (as in Figure \ref{fig:poset-example}). 
Let $M$ be the indecomposable sincere module $M(\Quiver)$ described in Definition~\ref{def:representation defined by a subquiver}.

\begin{enumerate}[(1)]
    \item Following \cite[Definition 3.1]{Kan22}, the
\emph{specialized rank matrix} of $M$, 
denoted by $\nabla(M)$, is defined as 

\[\nabla (M):=\begin{pmatrix} \mathfrak{R}(M) & -\mathfrak{R}_1(M) \\ {_0}\mathfrak{R}(M) & -{_0}\mathfrak{R}_1(M) \end{pmatrix}\]
where
\begin{itemize}

\item $\mathfrak{R}(M)$ is the number of submodules of $M$,

\item $\mathfrak{R}_1(M)$ is the number of submodules of $M$ containing the rightmost vertex of 
$\Quiver$,

\item ${_0}\mathfrak{R}(M)$ is the number of submodules of $M$ not containing the leftmost vertex of 
$\Quiver$, 

\item ${_0}\mathfrak{R}_1(M)$ is the number of submodules of $M$ containing the rightmost vertex but not the leftmost vertex of 
$\Quiver$.
\end{itemize}

\item We also recall the {\em dual specialized rank matrix} $\Delta(M)$ of $M$, \cite[Definition 4.1]{Kan22}:

\begin{equation*}
    \Delta(M) = \begin{pmatrix} \mathfrak{R}_1(M) & \mathfrak{R}_0(M) \\ {_0}\mathfrak{R}_1(M) & {_0}\mathfrak{R}_0(M) \end{pmatrix}, 
\end{equation*} 
where 
\begin{itemize}
\item $\mathfrak{R}_0(M)$ is the number of submodules of $M$ not containing the rightmost vertex of 
$\Quiver$, 
\item ${_0}\mathfrak{R}_0(M)$ is the number of submodules of $M$ not containing the leftmost vertex, nor the rightmost vertex of $\Quiver$. 
\end{itemize}
\end{enumerate}
\end{definition}

\begin{remark}
In \cite{Kan22}, rank matrices are defined using indeterminates. However, for our purposes it is enough to consider the specialized rank matrices as defined above. 
\end{remark}

\vspace{2mm}

\begin{example}~\label{ex:bigmatrix}
Consider the sincere indecomposable representation $M(\Quiver)$ 
of the following quiver of type $A_5$ (as in Example~\ref{ex:fence}).   
\[
\xymatrix{1\ar[r] & 2\ar[r] & 3 & 4\ar[l] \ar[r] & 5}
\] 
Then the specialized and the dual specialized rank matrices are 
\[
\nabla(M)=\begin{pmatrix} 11 & -7 \\ 
8 & -5 \end{pmatrix}, 
\hskip1cm
\Delta(M)=\begin{pmatrix} 7 & 4 \\ 
5 & 3 
\end{pmatrix}.
\]
\end{example}

One of the remarkable tools of the paper~\cite{Kan22} is that 
one can compute a rank matrix for a large poset by multiplying rank matrices of two smaller substructures, as we now explain.

Let $\Quiver$ be a quiver of type $A$, let $\Quiver_{(1)}$ and $\Quiver_{(2)}$ be subquivers such that $\Quiver$ is obtained by linking the two subquivers with an arrow as follows $\Quiver=\Quiver_{(1)} \searrow \Quiver_{(2)}$. By this, we mean that there is an arrow in $\Quiver$ which goes from the right most vertex of $\Quiver_{(1)}$ to the left most vertex of $\Quiver_{(2)}$ so that the resulting quiver is $\Quiver$ itself. 
Similarly, let 
$\Quiver_{(3)}$ and $\Quiver_{(4)}$ be subquivers such that $\Quiver$ is obtained by linking the two subquivers with an arrow as follows $\Quiver=\Quiver_{(3)} \swarrow \Quiver_{(4)}$. By this, we mean that there is an arrow in $\Quiver$ which goes from the left most vertex of $\Quiver_{(4)}$ to the right most vertex of $\Quiver_{(3)}$ so that the resulting quiver is $\Quiver$ itself. 
See Figure~\ref{fig:glue-Qs} for an example of both situations.

\begin{proposition}[Propositions 3.2, 4.2 of~\cite{Kan22}]\label{prop:order} 
Let $\Quiver$ be a quiver of type $A$ and let $M=M(Q)$ be the sincere indecomposable representation of $\Quiver$.

\begin{enumerate}[(1)] 
\item \label{prop:order:itm1}
Let $\Quiver=\Quiver_{(1)} \searrow \Quiver_{(2)}$. 
Let $M_1$ and $M_2$ be the sincere indecomposable representations of $\Quiver_{(1)}$ and $\Quiver_{(2)}$, respectively. Then there is an exact sequence 
$0\to M_2\to M\to M_1\to 0$ and 
\[
\nabla(M) = 
\nabla(M_1)\nabla(M_2).
\]

\item \label{prop:order:itm2} 
Let 
$\Quiver=\Quiver_{(3)}\swarrow \Quiver_{(4)}$. 
Let $M_3$ and $M_4$ be the sincere indecomposable representations of $\Quiver_{(3)}$ and $\Quiver_{(4)}$, respectively. Then there is an exact sequence 
$0\to M_3\to M\to M_4\to 0$ and 
\[
\Delta(M) = \Delta(M_3)  \Delta(M_4). 
\]
\end{enumerate}
\end{proposition}

The matrices $\nabla(M)$ and $\Delta(M)$ contain the same information, as we can obtain one from the other. The connection between the matrices $\nabla(M)$ and $\Delta(M)$ is as follows. 

\begin{lemma}[Lemma 4.5 of~\cite{Kan22}]\label{lm:nabla-delta}
Let $\Quiver$ be a quiver of type $A$ and let $M$ be the sincere indecomposable representation of $\Quiver$. Then one has: 
\[
\Delta(M)  =  \nabla(M) \begin{pmatrix}
        0 & 1 \\ -1 & 1 
\end{pmatrix} 
\quad
\mbox{ and }
\quad
\nabla(M)  =  \Delta(M) \begin{pmatrix}
        1 & -1 \\ 1 & 0 
\end{pmatrix}.
\]
\end{lemma}

\begin{example} \label{ex:matrix-example}

Let $M$ be the indecomposable sincere representation of the quiver $\Quiver$ from Figure~\ref{fig:glue-Qs}. 

In order to compute the number of submodules of $M$ we compute the specialized rank matrix of $M$. To do so, we consider $\Quiver=\Quiver_{(1)}\searrow \Quiver_{(2)}$ as in Figure~\ref{fig:glue-Qs}. 
Let $M_1$ and $M_2$ be the 
indecomposable sincere modules
for $\Quiver_{(1)}$ and for $\Quiver_{(2)}$ respectively. 


Computing the specialized rank matrices for $M_1$ and $M_2$ and using Proposition~\ref{prop:order}\ref{prop:order:itm1}, we have 
\[
\nabla(M)=
\nabla(M_1)
\nabla(M_2) 
=\begin{pmatrix} 3 & -2 \\ 
2 & -1 \end{pmatrix} \begin{pmatrix} 5 & -3 \\ 
2 & -1 \end{pmatrix}=\begin{pmatrix} 11 & -7 \\ 
8 & -5 \end{pmatrix}.
\]
The entry in position $(1,1)$ of each of the rank matrices is the number of submodules of the corresponding indecomposable sincere modules. So we found that $M$ has $11$ submodules. This number is confirmed by list of $11$ order ideals of the fence poset (corresponding to this quiver) given in Example~\ref{ex:fence}. 

Alternatively, we use  $\Quiver=\Quiver_{(3)}\swarrow \Quiver_{(4)}$ as in Figure~\ref{fig:glue-Qs} (the first is $1\to 2 \to 3$, and the second is $4\to 5$), $M_3$ and $M_4$ the indecomposable sincere modules for these quivers,  and 
Proposition~\ref{prop:order}\ref{prop:order:itm2}, to get:  
\[
\Delta(M)=
\Delta(M_3)
\Delta(M_4) 
=
\begin{pmatrix}
 3 & 1 \\ 
 2 & 1 
\end{pmatrix}
\begin{pmatrix}
2 & 1 \\ 
1 & 1 
\end{pmatrix}
=\begin{pmatrix} 7 & 4 \\ 
5 & 3 \end{pmatrix}.
\]

One can check that indeed (as in Lemma~\ref{lm:nabla-delta})
\[
\Delta(M) = \nabla(M)\begin{pmatrix}
    0 & 1 \\ -1 & 1
\end{pmatrix} 
\mbox{ and }
\nabla(M) = \Delta(M)\begin{pmatrix}
    1 & -1 \\ 1 & 0
\end{pmatrix}
\]
\end{example}

\begin{figure}[ht]
\centering
\includegraphics[scale=.7]{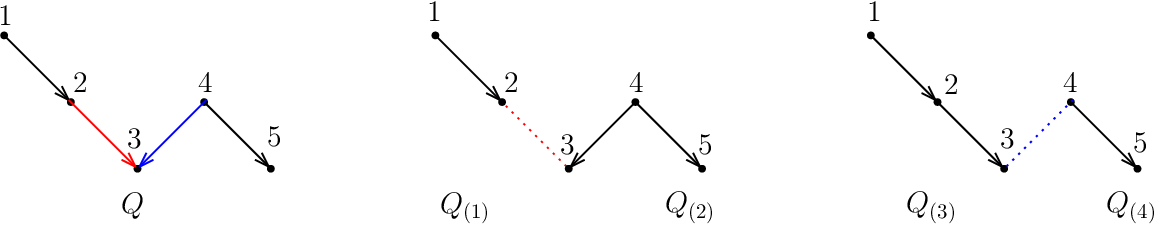}
\caption{The quiver $\Quiver$ is obtained from $\Quiver_{(1)}$ and $\Quiver_{(2)}$ by linking them with an arrow (in red) or from $\Quiver_{(3)}$ and $\Quiver_{(4)}$ with an arrow (in blue).}
\label{fig:glue-Qs}
\end{figure}

\subsection{From a twice-punctured disk to an annulus}\label{sec:D-to-A}

We will now associate a triangulation $T'$ of an annulus to any triangulation $T$ of the twice-punctured disk $(S,M)$ in such a way that a quiddity sequence of $T'$ is the same as the quiddity sequence of the boundary of $T$.

The annulus we construct will have exactly the same vertices on the outer boundary (as the boundary of $(S,M)$). 
The inner boundary will have marked points for $\punctureP$ and $\punctureQ$, (Figures~\ref{fig:set-up-disk-I} and \ref{fig:set-up-disk-II}), or only for one of them (Figure~\ref{fig:set-up-disk-III}), and possibly with additional marked points.

Informally speaking, we draw an arc connecting 
the two punctures $\punctureP$ and $\punctureQ$ and create an inner boundary component by cutting along this arc. 
Before we explain this, we observe that we can assume the triangulation $T$ has no peripheral arcs:

\begin{lemma}\label{lem:remove-peripheral} 
Let $\widetilde{T}$ be a triangulation of a twice-punctured disk $\widetilde{S}$, and 
let $T$ be the triangulation of a twice-punctured disk (possibly with fewer marked points) $S$ 
which is obtained by removing all triangles of $\widetilde{T}$ which contain peripheral arcs. (In particular, the new triangulation $T$ does not contain any ``ear triangles'' --- triangles with two boundary segments at the boundary.) 
Then the growth coefficient of the frieze from the boundary of $\widetilde{S}$ corresponding to $\widetilde{T}$ is equal to that of the frieze from the 
boundary of the triangulation $T$ of the new surface $S$.
\end{lemma}

\begin{proof}
This follows from repeatedly applying Theorem~\cite[Theorem 3.1]{BFPT19}. 
\end{proof}

From now on we always assume that the triangulation of the twice-punctured disk contains no peripheral arcs thanks to Lemma~\ref{lem:remove-peripheral}. 

\begin{notation}\label{notation:p-q}
Let $T$ be a triangulation of a twice-punctured disk without any peripheral arcs. 
We distinguish three main cases: 

\begin{enumerate}[(I)]
    \item \label{notation:p-q:item:I} there is no arc connecting 
    the two punctures $\punctureP$ and $\punctureQ$,
    \item\label{notation:p-q:item:II} 
    there is an arc connecting 
    $\punctureP$ and $\punctureQ$, and both $\punctureP$ and $\punctureQ$ have arcs connected to the boundary of the disk,
    \item\label{notation:p-q:item:III} 
    there is an arc connecting
    $\punctureP$ and $\punctureQ$, and only one of them has arcs connected to the boundary.
\end{enumerate}
\end{notation}

We are now ready to introduce some notation, see Figures~\ref{fig:set-up-disk-I}, \ref{fig:set-up-disk-II} and \ref{fig:set-up-disk-III}. 

\begin{enumerate}[(a)]
\item 
As before, we label the marked points by $1,\dots, n-2$ counterclockwise around the boundary and denote the punctures by $\punctureP$ and $\punctureQ$. 
Let $p$ be the number of arcs incident to $\punctureP$ in a small neighborhood of $\punctureP$: if we cut the surface with a small disk centered at $\punctureP$, then $p$ is the number of arcs ending at $P$ in that disk. Similarly, let $q$ be the number of arcs in a small neighborhood of $\punctureQ$.
\item
In Case \ref{notation:p-q:item:I}, we denote the arcs starting at the boundary and ending at the puncture $\punctureP$ by $d_1,\dots, d_p$ (where $p\ge 1$), counterclockwise around the boundary, with $d_p$ ending at the vertex $1$ on the boundary. Similarly, we denote the arcs starting at the boundary and ending at $\punctureQ$ by $e_1,\dots, e_q$ (where $q\ge 1$), counterclockwise around the boundary, with $e_1$ starting at vertex $w$, where $1\le w\le n-p$. 
The remaining arcs all cross the line connecting $\punctureP$ and $\punctureQ$. We denote these arcs of $T$ by $c_1,\dots, c_m$ , in the order they appear when going from $\punctureP$ to $\punctureQ$. Note that $m+p+q=n+1$, the total number of arcs (cf. Remark~\ref{rem:number-of-arcs}). 
See left hand picture of Figure~\ref{fig:set-up-disk-I}.
\item
For \ref{notation:p-q:item:II} and \ref{notation:p-q:item:III} we denote the arcs at $\punctureP$ or at $\punctureQ$ as indicated in Figures~\ref{fig:set-up-disk-II} and~\ref{fig:set-up-disk-III}. Note that in case \ref{notation:p-q:item:III} the loop arc around the radius connecting $\punctureP$ and $\punctureQ$ is counted twice by definition. 
\end{enumerate}

In Case \ref{notation:p-q:item:I}, 
we have $p,q\ge 1$. 
In Case \ref{notation:p-q:item:II}, 
we have $p,q\ge 3$. In Case 
\ref{notation:p-q:item:III}, 
we have $p\ge 6$ while $q=1$, or $q=1$ while $q\ge 6$ (as $n\ge 4$); we will always assume $p\ge 6$ and $q=1$ when considering case
\ref{notation:p-q:item:III}.

\begin{figure}[htbp!]
\centering
\includegraphics[scale=.5]{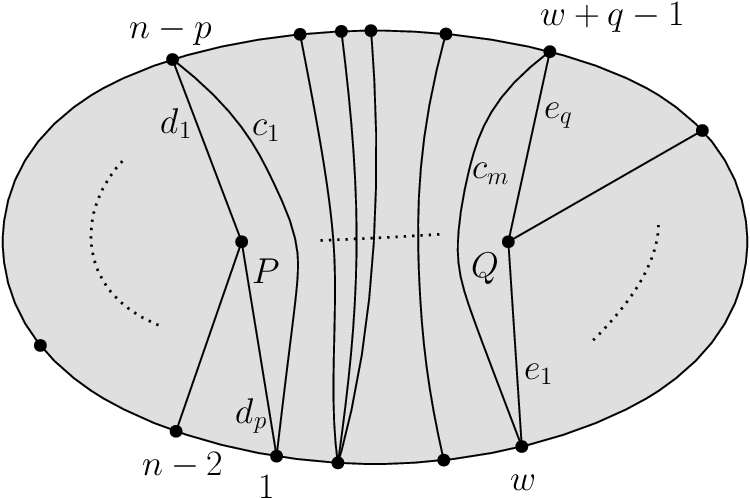}
\hskip.5cm
\includegraphics[scale=.5]{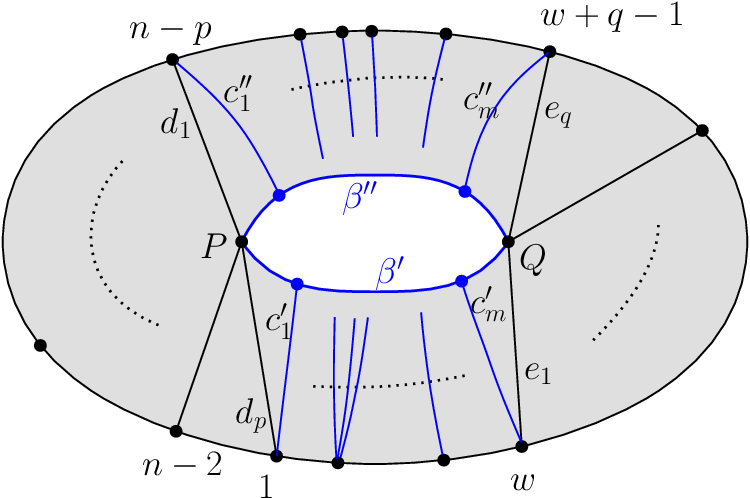}
\caption{Case \ref{notation:p-q:item:I}: $T$ has no arc connecting $\punctureP$ and $\punctureQ$.}
    \label{fig:set-up-disk-I}
\bigskip
\centering
\includegraphics[scale=.5]{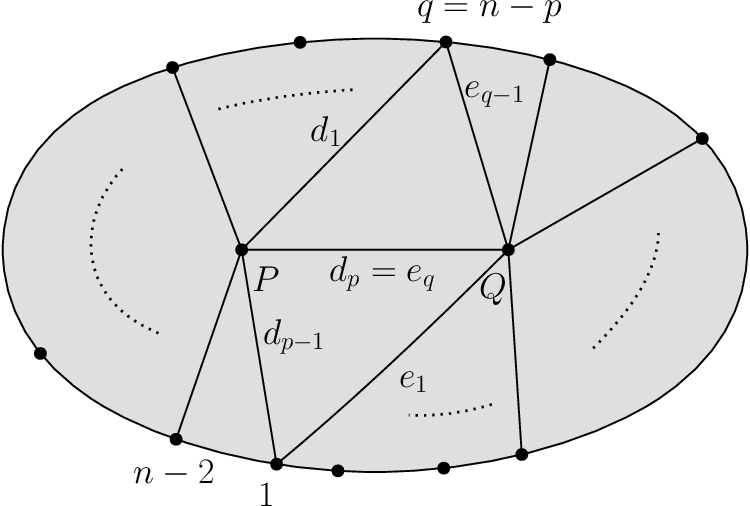}
\hskip.5cm
\includegraphics[scale=.5]{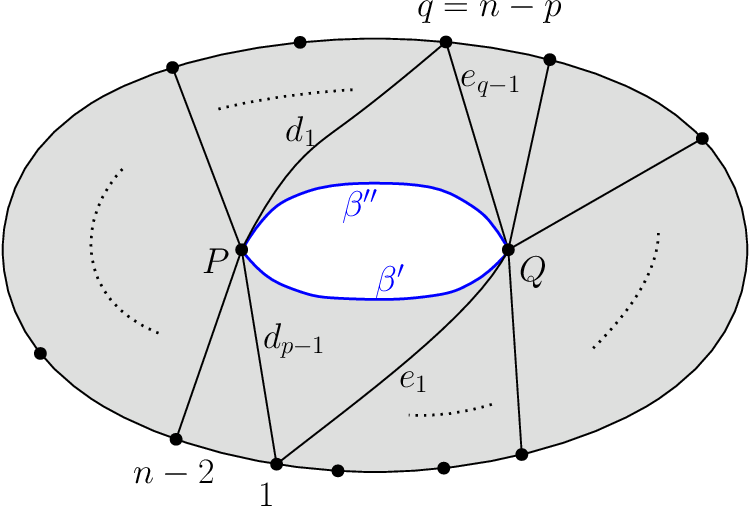}
    \caption{Case \ref{notation:p-q:item:II}: $T$ has an arc connecting $\punctureP$ and $\punctureQ$ but no loop.}
    \label{fig:set-up-disk-II}
\bigskip
\centering
\includegraphics[scale=.5]{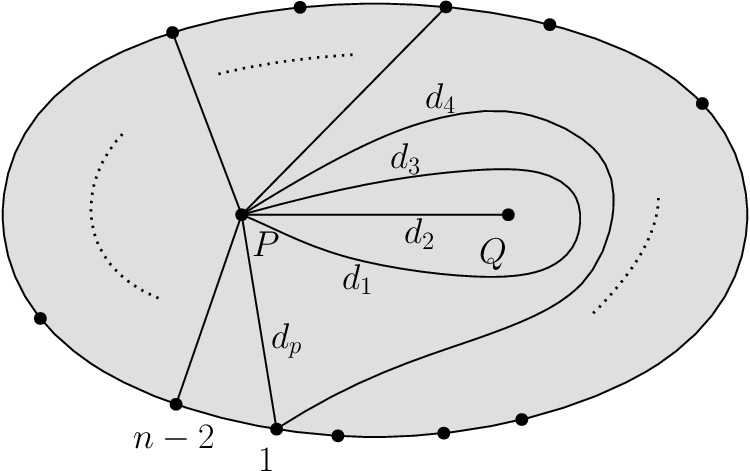}
\hskip.5cm
\includegraphics[scale=.5]{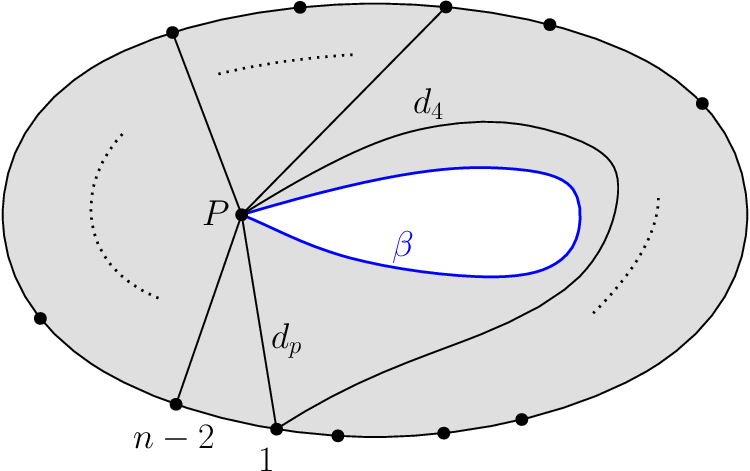}
    \caption{Case \ref{notation:p-q:item:III}: $T$ has a self-folded triangle with radius connecting $\punctureP$ and $\punctureQ$.}
    \label{fig:set-up-disk-III}
\end{figure}

Let us point out that, in Case \ref{notation:p-q:item:I}, there may be multiple arcs among the $c_i$ starting at vertex $1$ or vertex $n-p$.

We are now ready to describe the triangulated annulus $T'$ associated to $T$. 
We detail this process case by case:

\noindent
\underline{Case \ref{notation:p-q:item:I}}: 
In order to create an inner boundary component, 
we draw a curve $\beta'$ from $\punctureP$ to $\punctureQ$ following the boundary segments along the vertices $1,2,\dots,w$, and $\beta''$ from $\punctureQ$ to $\punctureP$ following the boundary segments along the vertices $w+q-1$
to $n-p$. See the right hand side of Figure~\ref{fig:set-up-disk-I}. 

In order to achieve the same quiddity sequence on the outer boundary of $T'$ as that on the boundary of $T$, we do the following: 
Each arc $c$ starting at one of the points of $\{1,2,\dots, w\}$ and ending at one of the points of $\{w+q-1,\dots, n-p\}$ gets replaced by two arcs $c'$ and $c''$ by cutting out a middle segment and creating vertices on $\beta'$ and $\beta''$ which are the endpoints of $c'$ and $c''$ respectively, i.e. the arc $c'$ ends where $c$ would meet $\beta'$ and $c''$ where $c$ would meet $\beta''$.

Every fan of arcs of $T$ at a boundary point $i$ in  $\{1,2,\dots,w\}$ is of the form 
\[(d_p), c_t,c_{t+1},\dots, c_{t+s_t}, (e_1)
\] 
for some $t\in\{1,\dots, m\}$ and $s_t\ge 0$ (with $t+s_t\le m$). 
We wrote $d_p$ and $e_1$ in parentheses because $d_p$ only occurs for $i=1$ and $e_1$ only occurs for $i=w$; if $i=1=w$, then both $d_p$ and $e_1$ occur.
On the annulus, we replace this fan by the fan $(d_p), c_t',\dots, c_{t+s_t}',(e_1)$ at the corresponding point on the outer boundary of the annulus. 

Similarly, every fan of arcs of $T$ at one of the points $\{w+q-1,\dots, n-p\}$ is of the form $(e_q), c_t,c_{t-1},\dots, c_{t-s_t-1}, (d_1)$. 
It gets replaced by the fan of arcs 
$(e_q),c_t'',c_{t-1}'',\dots, c_{t-s_t-1}'',(d_1)$ at the corresponding point on the outer boundary of the annulus. 

What we currently have
is not a triangulation in general. 
Between any two 
fans at neighboring vertices $i$ and $i+1$ on the outer boundary, we have a quadrilateral formed by: the boundary segment between $i$ and $i+1$, the first arc in the fan at $i+1$, a segment on $\beta'$ (or on $\beta''$), and the last arc of the fan at $i$. 
For example, in the following figure (which is in an excerpt of the triangulation of Figure~\ref{fig:set-up-disk-I}), we see three quadrilaterals; between the fans at vertices $1$ and $2$ of the lower boundary, 
between the fans at vertices $2$ and $3$, and between the fans at vertices $3$ and $w$: 
\[
\includegraphics[scale=.7]{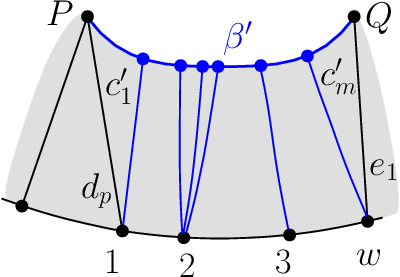}
\]
To produce a triangulation, in the final step we identify the endpoints of the leftmost arc of the fan at $i+1$ and of the rightmost arc of the fan at $i$, turning every quadrilateral into a triangle. 
In our example, this final process removes three vertices from the upper boundary, as shown in the figure below: 
\[
\includegraphics[scale=.7]{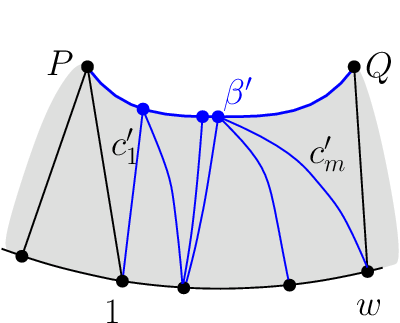}
\]

\noindent
\underline{Case \ref{notation:p-q:item:II}}: 
$T$ contains the arc connecting $\punctureP$ and $\punctureQ$. 
and we have $p,q\ge 3$
We replace the arc connecting $\punctureP$ and $\punctureQ$ by two curves $\beta'$ and $\beta''$ to form the inner boundary. See the right hand side of Figure~\ref{fig:set-up-disk-II}. 

\noindent
\underline{Case \ref{notation:p-q:item:III}}:
$T$ contains a self-folded triangle such that its radius is the arc $d_2$ connecting $\punctureP$ and $\punctureQ$, and the loop of the self-folded triangle is at $\punctureP$ (left hand picture in Figure~\ref{fig:set-up-disk-III}).
In this case, we cut out the self-folded triangle from the surface; the puncture $\punctureQ$ and the radius $d_2$ are removed as a result. 
The loop arc becomes the boundary $\beta$ of the newly created inner boundary. 
See the right hand picture in Figure~\ref{fig:set-up-disk-III}.

In all three cases, we have constructed an annulus $A=A(S)$, with triangulation $T'$, whose outer boundary defines the same quiddity sequence as the boundary of the triangulation $T$ of the twice-punctured 
disk $S$. 
Note that the new triangulation $T'$ does not have any peripheral arcs.

We have the following analog of Lemma~\ref{lm:quid-peripheral}.

\begin{lemma}\label{lm:quid-peripheral-ann}

Consider a peripheral arc of the form $\gamma_{i-1,i+1}=\gamma_{i-1,i+1}^0$ 
in 
the annulus $A=A(S)$ (reducing endpoints modulo $n-2$) and let $r'_i$ be the number of arcs of $T'$ incident with $i$. Then we have 
\[
\rho(M(\gamma_{i-1,i+1})) = r'_i+1 
\]
In particular, if $F$ is the infinite frieze from the tube $\mathcal{T}_1$ corresponding to the boundary of $T$, and  
if $F'$ is the frieze from the outer boundary of the triangulated annulus $T'$, then the two friezes have the same quiddity sequence.  
\end{lemma}

\begin{proof}
Since $T'$ has no peripheral arcs, every arc in $T'$ has endpoints at both the outer and inner boundary. As in case \ref{rem:module-arc:itm:a} in the proof of Lemma~\ref{lm:quid-peripheral}, the quiver corresponding to $\gamma_{i-1,i+1}$ is an equioriented type $A$ quiver with $r_i'$ vertices, and the result follows.
\end{proof}

%

\subsection{From peripheral arcs to string modules}\label{sec:arc-gives-string}



Let $T'$ denote the triangulation of  the annulus $A$ constructed in Section~\ref{sec:D-to-A}, and let $\AnnulusQuiver=\Quiver_{T'}$ denote the quiver of $T'$.
In this section we explain how to associate peripheral arcs in 
the annulus $A$ to string representations of $\AnnulusQuiver$, following \cite{ABCP10, BZ11}.

\begin{remark}\label{rem:shape-Q-T-annulus}
Suppose our annulus has $n$ marked points on a boundary component and $m$ marked points on the other. 
Since our triangulation $T'$ has no peripheral arc, the quiver $\AnnulusQuiver$ 
of the triangulated annulus $T'$ consists of a cyclic graph with $n$ arrows pointing clockwise and $m$ arrows pointing counterclockwise.
\end{remark}

Consider any non self-crossing peripheral arc $\gamma$. 
Since our triangulation $T'$ of the annulus   contains no peripheral arcs, $\gamma$ crosses each arc of $T'$ at most once. 
The quiver $\Quiver_\gamma$ of $\gamma$, defined in Section~\ref{sec:quiver}, is a subquiver of $\AnnulusQuiver$ whose vertices correspond to the distinct arcs of $T'$ crossed by $\gamma$ and whose arrows correspond to the distinct angles of $T'$ crossed by $\gamma$.  Let 
$\alpha_1, \dots, \alpha_{\ell+1}$ denote these arcs and let
$u_1,\dots,u_\ell$ denote these angles of $T'$; let $1,\dots, \ell+1$ denote the vertices in $\AnnulusQuiver$ which correspond to the $\alpha_i$. 
The fact that $\gamma$ does not cross itself and the fact that it doesn't cross any arc of $T$ twice tell us that $\Quiver_\gamma$ is of type $A_{\ell+1}$.

In Definition \ref{def:representation defined by a subquiver}, we define a module $M$ using a subquiver of a type $A$ quiver. We now define the module $M(\Quiver_\gamma)$ in similar manner: at each vertex of $\Quiver_\gamma$, we have one copy of the field; and at each arrow of $\Quiver_\gamma$, we have the identity map. The rest of the vertices and the arrows of the quiver $\AnnulusQuiver$ of the triangulated annulus are assigned the zero vector space and the zero map.

For each peripheral arc $\gamma$ which does not cross itself, we define $M(\gamma)$ to be the module $M(\Quiver_\gamma)$. 

Note that $M(\gamma)$ is a type of nicely-behaved module called a \emph{string module} \cite[Page 160]{BR87}. In general string modules are allowed to have simples appearing with multiplicity more than one, but $M(\gamma)$ has simples appearing with multiplicity at most one.

A string module arises from a sequence, or word (called a \emph{string})
whose letters are arrows or formal inverses of arrows of $\AnnulusQuiver$. 
In our case, our word corresponds to the angles of $T'$ crossed by $\gamma$: 
$u_1 u_2\cdots u_\ell$. 
We identify this word with the corresponding arrows and inverse arrows of $\Quiver_\gamma$; that is, we identify the string with the quiver $\Quiver_\gamma$.

The support of $M(\gamma)=M(\Quiver_\gamma)$ is precisely $\Quiver_\gamma$. Since $\Quiver_\gamma$ is 
a type $A_{\ell+1}$ quiver, we can apply the formulas for computing the number of submodules of a type $A$ sincere module to compute the number of submodules of $M(\gamma)$. 

\subsection{From a band module to the growth coefficient for the annulus}
\label{sec:quasi-simple band module}

We give an overview on how to associate a loop in the interior of our annulus to a certain quasi-simple band module whose number of submodules give us the growth coefficient of friezes arising from $T'$.

Let $\varepsilon$ denote a non-contractible closed curve in the annulus which has no self-crossings and is disjoint from the boundary; note that $\varepsilon$ is called a ``bracelet'' $\text{Brac}_1$ in \cite{MSW13, GMV19}.
The quiver $\Quiver_\varepsilon$ of the closed curve $\varepsilon$ is exactly the quiver $\AnnulusQuiver=\Quiver_{T'}$ 
of our triangulated annulus (cf. Remark~\ref{rem:shape-Q-T-annulus}). 
A cyclic string (called band) corresponding to $\varepsilon$ is the sequence we get by following $\Quiver_\varepsilon$, going counterclockwise around the annulus; because the sequence is cyclic, we can start at any vertex of $\Quiver_\varepsilon$, and for convenience we will start at the vertex of $\Quiver_\varepsilon$ associated to the arc $d_1$. 

To $\varepsilon$, we can associate a family of indecomposable modules  called band modules, \cite[Proposition 4.3]{ABCP10} (see also \cite[Theorem 1.1(2)]{BZ11}). From this family we pick a specific 
quasi-simple module which we  denote $M(\varepsilon)$; it has one copy of the field at each vertex of $\Quiver_\varepsilon$ and the identity map for every arrow in $\Quiver_\varepsilon$.

Let $\rho(M(\varepsilon))$ denote the number of submodules of $M(\varepsilon)$. The submodules of $M(\varepsilon)$ correspond to the down-closed subsets of $\Quiver_\varepsilon$, so $\rho(M(\varepsilon))$ is equal to the number of down-closed subsets of $\Quiver_\varepsilon$.

In \cite[Definitions 3.14, 3.16]{MSW13}, the authors associate to $\varepsilon$ and $T'$ a Laurent polynomial $x(\varepsilon)=x_{T'}(\varepsilon)$ with positive coefficients; let $f(\varepsilon)$ denote the number of terms in the Laurent polynomial $x(\varepsilon)$. 
In \cite[Theorem 5.7]{MSW13}, they show that $f(\varepsilon)$ is equal to the number of down-closed subsets of $Q_\varepsilon$. 
Therefore, $f(\varepsilon)=\rho(M(\varepsilon))$. This connection between the terms in $x(\varepsilon)$ and the submodules of $M(\varepsilon)$ is discussed in \cite[Remark 5.8, Figure 17]{MSW13} and \cite[Remark 11.7, Section 12]{GLS22}. 
(A much stronger result which connects $x(\varepsilon)$ and  $M(\varepsilon)$ via the CC-map is given in \cite[Section 11]{GLS22}.)

It was shown in \cite[Proposition 6.1]{GMV19} that the growth coefficient of the two friezes arising from $T'$ is given by $f(\varepsilon)$, the number of terms in the Laurent polynomial $x(\varepsilon)$. 
Therefore, we can conclude the following:

\begin{proposition}\label{prop:growth coefficient equals number of submodules of epsilon}
Let $T'$ be the triangulation of  the annulus $A$ constructed in Section~\ref{sec:D-to-A}.
Let $\quid$ be the quiddity sequence of the outer boundary. 
Then we have 
\begin{equation*}
s(\quid)=
\rho(M(\varepsilon)), 
\end{equation*}    
that is, the growth coefficient of the frieze determined by $\quid$ is equal to the number of submodules of $M(\varepsilon)$.  
\end{proposition}

In the next section we will use the results of \cite{Kan22} (on posets arising from quivers of affine type $A$) to compute $\rho(M(\varepsilon))$.

\subsection{Computing the growth coefficient for the annulus}\label{sec:outer-boundary}

In this section we compute the growth coefficient for the triangulation $T'$ of  
the annulus $A$ constructed in Section~\ref{sec:D-to-A}. 

In case \ref{notation:p-q:item:I}, we define $\gamma'$ to be the peripheral arc from $\punctureP$ to $\punctureQ$ following the inner boundary along $\beta'$. Similarly, let $\gamma''$ be the peripheral arc from $\punctureQ$ to $\punctureP$ following the inner boundary along $\beta''$. See Figure~\ref{fig:gammas-epsilon}. 

As in Section~\ref{sec:arc-gives-string}
, we use $\gamma'$ and $\gamma''$ to define strings 
which give rise to
representations $M(\gamma')$ and $M(\gamma'')$ for $\AnnulusQuiver$.

Recall that our triangulation $T'$ is constructed from a triangulation $T$ of a twice-punctured disk and that the arcs $\gamma'$ and $\gamma''$ both come from a common arc $\gamma$  in the twice-punctured disk: the arc connecting the two punctures (see section \ref{sec:arc-model}). 
Therefore by construction, the quivers of $\gamma'$ and $\gamma''$ are isomorphic as abstract directed graphs. Thus 
\[\rho(M(\gamma'))= \rho(M(\gamma'')).\]

\begin{notation}\label{notation:value-a}
In case \ref{notation:p-q:item:I}, let 
 $\gamma'$ and $\gamma''$ be as defined above, and we set $\ag=\rho(M(\gamma'))= \rho(M(\gamma''))$. 
\end{notation}

\begin{remark}\label{rem:a-is-1}
In cases \ref{notation:p-q:item:II} and \ref{notation:p-q:item:III}, there is an arc connecting the two punctures and so the analog of $\gamma'$ and $\gamma''$ in the annulus is homotopic to a segment of the inner boundary (or to the entire inner boundary). The curves therefore correspond to the $0$ module and we set $\ag=1$ in cases \ref{notation:p-q:item:II} and \ref{notation:p-q:item:III}.
\end{remark}

For the proofs in this section, we will use the following:

\begin{notation} 
\begin{enumerate}
    \item[(1)] If $\gamma$ is an  
arc or closed curve in the triangulated annulus, by abuse of notation we will write $\gamma$ to denote a word (string or band) corresponding to $\gamma$ and also to denote the string or band module $M(\gamma)$ corresponding to $\gamma$. For example, we will simply write $\nabla(\gamma)$ instead of $\nabla(M(\gamma))$ when writing the specialized rank matrices defined in Definition \ref{def:rm}.
We will also write $\gamma$ for the quiver $\Quiver_\gamma$ of $\gamma$.

\item[(2)] For $t\ge 0$, 
we write $U_t$ to denote an inverse segment of $t$ arrows and $D_t$ to denote a direct segment with $t$ arrows: 
\[\resizebox{13cm}{!}{$U_t=$\xymatrix{
      & & & t+1 \ar[dl] \\
   &   & t \ar@{.>}[dl]  &  \\
   &  2 \ar[dl]   &  & \\
   1 &&&  
   }\hspace{4cm}$D_t=$ \hspace{.5cm}
   \xymatrix{
 1 \ar[dr]      & & & \\
   & 2 \ar@{.>}[dr]  &   &  \\
   &            & t \ar[dr] & \\
   &&& t+1 }
   }\]
   \end{enumerate}
\end{notation}

\begin{lemma}\label{lm:Us-Ds}
Using Definition~\ref{def:rm}, for every $t\ge 0$, one has 
\[
\nabla(U_t)= 
\begin{pmatrix}
    t+2 & -1 \\ 1 & 0 
\end{pmatrix}, \quad 
\nabla(D_t)= 
\begin{pmatrix}
    t+2 & -t-1 \\ t+1 & -t
\end{pmatrix},
\]
\[
\Delta(U_t)= 
\begin{pmatrix}
    1 & t+1 \\ 0 & 1
\end{pmatrix}, 
\quad 
\Delta(D_t)= 
\begin{pmatrix}
    t+1 & 1 \\ t & 1 
\end{pmatrix}. 
\]
\end{lemma}

\begin{corollary}\label{cor:det-1}
For any string module $M$, we have 
\[
\det\nabla(M)=1 \mbox{ and }
\det\Delta(M)=1.
\]
\end{corollary}

\begin{proof}
The determinant of each of the matrices $\nabla(U_t)$, $\nabla(D_t)$, $\Delta(U_t)$ and $\Delta(D_t)$ is $1$ and by Proposition~\ref{prop:order} we get the desired result.
\end{proof}

The inverse of a string $\beta$, denoted $\beta^{-1}$, is obtained by reading $\beta$ from right to left; inverse segments are turned into direct segments and vice versa. The inverse of a type $A$ quiver is defined in a similar manner.
\begin{lemma}\label{lm:nabla-inverse} 
Let $\beta$ be a peripheral arc such that $M(\beta)$ is a rigid indecomposable representation (of the quiver $\AnnulusQuiver$ of the triangulation $T'$). 

Let $\alpha=\beta^{-1}$, and 
let $a,b,c,d\ge 0$ be such that 
$\nabla(\beta)=
\begin{pmatrix} \mathfrak{R}(\beta) & -\mathfrak{R}_1(\beta) \\ {_0}\mathfrak{R}(\beta) & -{_0}\mathfrak{R}_1(\beta) \end{pmatrix}
 =
 \begin{pmatrix}
a & -b \\ c & -d 
\end{pmatrix}$.

\noindent
Then we have 
\[
\nabla(\alpha)  =  
\begin{pmatrix}
a & c-a \\ a-b & c+b-a-d \\
\end{pmatrix} 
\]

\end{lemma}

\begin{proof}
By Definition~\ref{def:rm}, 
\begin{equation*}
\nabla(\alpha) = \begin{pmatrix} \mathfrak{R}(\alpha) & -\mathfrak{R}_1(\alpha) \\ {_0}\mathfrak{R}(\alpha) & -{_0}\mathfrak{R}_1(\alpha) \end{pmatrix}.
\end{equation*}
Now 
\begin{align*}
\mathfrak{R}(\alpha) & =\mathfrak{R}(\beta) = a, \\
\mathfrak{R}_1(\alpha) & = {_1}\mathfrak{R}(\beta) = \mathfrak{R}(\beta) - {_0}\mathfrak{R}(\beta) = a-c, \\
_0\mathfrak{R}(\alpha) & = \mathfrak{R}_0(\beta) = \mathfrak{R}(\beta) - \mathfrak{R}_1(\beta) = a-b,\\
{_0}\mathfrak{R}_1(\alpha) & =  {_1}\mathfrak{R}_0(\beta) = {_1}\mathfrak{R}(\beta) - {_1}\mathfrak{R}_1(\beta) = 
\left(\mathfrak{R}(\beta) - {_0}\mathfrak{R}(\beta) \right)
- \left( \mathfrak{R}_1(\beta) - {_0}\mathfrak{R}_1(\beta)\right)\\
 & = a -c -b +d, 
\end{align*}
\end{proof}

Recall the definitions of $p$ and $q$ from Notation~\ref{notation:p-q}. In case \ref{notation:p-q:item:I}, there are $p$ arcs incident with the puncture $\punctureP$ and $q$ arcs incident with the puncture $\punctureQ$.

\begin{proposition}\label{prop:growth-I} 
Let $T'$ be the triangulation of  the annulus $A$ constructed in Section~\ref{sec:D-to-A}.
Let $\quid$ be the quiddity sequence of the outer boundary. 
Then the growth coefficient of the frieze determined by the outer boundary of this triangulated annulus is 
\begin{equation*}
s(\quid)=
\ag^2pq-2, 
\end{equation*}
where $\ag=\rho(M(\gamma'))=\rho(M(\gamma''))$ as in Notation~\ref{notation:value-a}.
\end{proposition}

\begin{proof}

Due to Proposition \ref{prop:growth coefficient equals number of submodules of epsilon}, it suffices to show that $\rho(M(\varepsilon))=\ag^2pq-2$.

In case \ref{notation:p-q:item:I}, the triangulation $T'$ and $\varepsilon$ are as illustrated in Figure~\ref{fig:gammas-epsilon}.

\begin{figure}[htbp!]
    \centering
    \includegraphics[scale=.6]{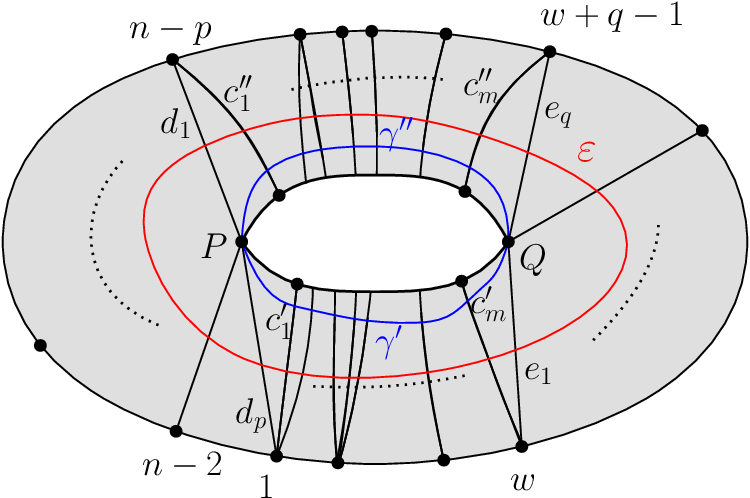}
    \caption{The arcs $\gamma'$, $\gamma''$ and the loop $\varepsilon$ in case \ref{notation:p-q:item:I}.}
    \label{fig:gammas-epsilon}
\end{figure}
So reading $\varepsilon$ as a string starting from $d_1$, going counter clockwise and ending again at $d_1$ we have the following 
quiver $\Quiver_\varepsilon$ ($*$ means that you glue the left to the right end): 
\[ 
* \underbrace{U_{p-1} \searrow \gamma' \searrow U_{q-1} \searrow \gamma''}_{=:\beta}\searrow *
\]

In case \ref{notation:p-q:item:II} (recalling that $p,q\ge 3$), we have the following picture: 
\[
\includegraphics[scale=.6]{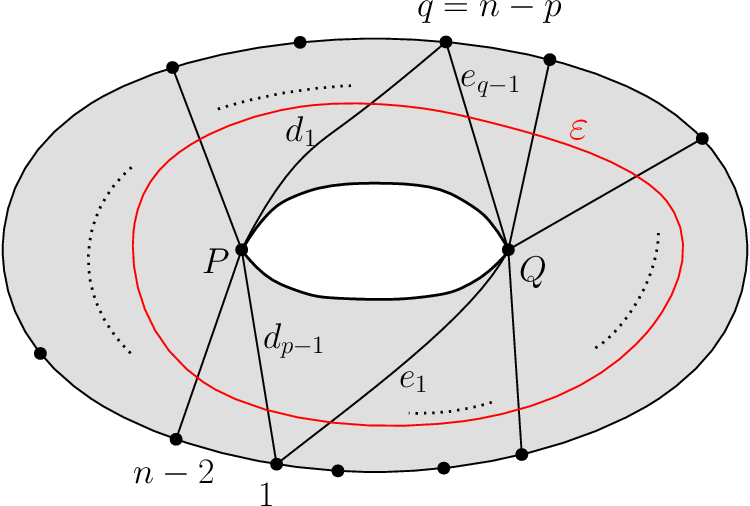}
\]
and so, reading $\varepsilon$ counterclockwise from $d_1$ as before, we have 
\[ 
* \underbrace{U_{p-2}\searrow U_{q-2}}_{=:\beta}\searrow *
\]

Finally, case \ref{notation:p-q:item:III} (with $p\ge 6$) is as follows 
\[\includegraphics[scale=.6]{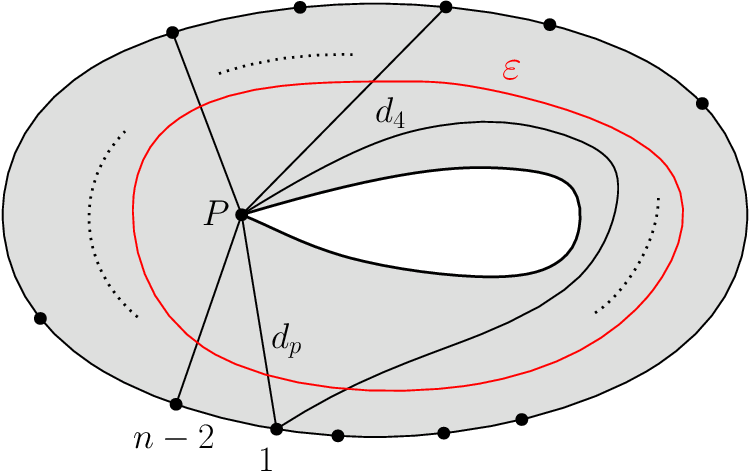}
\]
with $\varepsilon$ (reading counterclockwise, starting with $d_4$) 
\[ 
* \underbrace{U_{p-4} }_{\beta} \searrow *
\]

Let $\beta$ denote the (non-cyclic) string which is the subsequence of $\varepsilon$ obtained by reading $\varepsilon$ (that is, the quiver $\Quiver_\varepsilon$) clockwise starting from $d_1$ (or from $d_4$ in case \ref{notation:p-q:item:III}), reading all the arrows and inverse arrows of $\Quiver_\varepsilon$ except the last arrow $\searrow$. 
The quivers for $\beta$ and for $\varepsilon$ have the same vertices but they differ by one arrow. 
Following the notation of \cite[Section 3]{Kan22}, this 
means we can write $\varepsilon=\circlearrowright\beta$.  The circular symbol $\circlearrowright$ is used 
to take 
a type $A$ quiver and connect the two degree-one vertices using a new arrow to form  
a type affine $A$ quiver; in our scenario, we take $\beta$ and form $\varepsilon$ by adding a new arrow to $\beta$.  
The number $\rho(\circlearrowright\beta)$ is equal to the trace of $\nabla(\beta)$, by~\cite[Prop 3.2 (9)]{Kan22}.
We compute this trace case by case. 
Recall that $\ag=\rho(M(\gamma'))=\rho(M(\gamma''))$.

\begin{itemize}
\setlength{\itemindent}{0em}
    \item Case \ref{notation:p-q:item:I}: Let $\nabla(\gamma')=\begin{pmatrix}
\ag & -b \\ c & -d 
\end{pmatrix}$ with $\ag,b,c,d\ge 0$. 
As abstract directed graphs,  the quiver $\Quiver_{\gamma''}$ of $\gamma''$ is the inverse of the quiver $\Quiver_{\gamma'}$  of $\gamma'$, so 
we have 
$\nabla(\gamma'')= 
\begin{pmatrix}
\ag & c-\ag \\ \ag-b & c+b-\ag-d 
\end{pmatrix}$ by Lemma~\ref{lm:nabla-inverse}. 
Then 
\begin{eqnarray*}
\rho(\circlearrowright\beta) & = & \tr\left( \nabla(\beta) \right) 
 = 
\tr\left(\nabla(U_{p-1}\searrow \gamma' \searrow U_{q-1} \searrow \gamma'') \right)
\\
 & = & 
\tr\left(\nabla(U_{p-1})\nabla(\gamma')\nabla(U_{q-1})\nabla(\gamma'') \right) 
\text{ by Proposition \ref{prop:order}\ref{prop:order:itm1} }\\ 
 & = & 
\tr(\begin{pmatrix}
p+1 & -1 \\ 1 & 0 
\end{pmatrix}
\begin{pmatrix}
\ag & -b \\ c & -d     
\end{pmatrix}
\begin{pmatrix}
q+1 & -1 \\ 1 & 0     
\end{pmatrix} 
\begin{pmatrix}
\ag & c-\ag \\ \ag-b & c-\ag+b-d
\end{pmatrix} ) 
\text{ by Lemma \ref{lm:Us-Ds}}
\\
 & = & \tr(\begin{pmatrix}\ag^2pq + \ag^2q - \ag cq - bc + \ag d & * \\ 
 * & -\ag^2q + \ag cq - bc + \ag d
\end{pmatrix} ) \\
 & = & \ag^2pq -2(bc-\ag d) 
 \\
 & = &
 \ag^2pq-2(1) \text{ by Corollary \ref{cor:det-1}}
\end{eqnarray*}

\item Case \ref{notation:p-q:item:II}: We have $\ag=1$ (see Remark~\ref{rem:a-is-1}) and 
\begin{eqnarray*}
\rho(\circlearrowright\beta) & = & \tr(\nabla(\beta)) = 
\tr(\nabla(U_{p-2}\searrow  U_{q-2} ) )\\
 & = & 
\tr(\nabla(U_{p-2})\nabla(U_{q-2})) \text{ by Proposition \ref{prop:order}\ref{prop:order:itm1} } \\ 
 & = & 
\tr(\begin{pmatrix}
p & -1 \\ 1 & 0 
\end{pmatrix}
\begin{pmatrix}
q & -1 \\ 1 & 0     
\end{pmatrix}) \text{ by Lemma \ref{lm:Us-Ds}}
 \\
 & = & \tr(\begin{pmatrix}pq -1  & * \\ 
 * & -1
\end{pmatrix}) \\
 & = & pq -2 
\end{eqnarray*}

\item Case \ref{notation:p-q:item:III}: We again have $\ag=1$, with $q=1$. We get 
\begin{eqnarray*}
\rho(\circlearrowright\beta) & = & \tr(\nabla(\beta)) = 
\tr(\nabla(U_{p-4}))\\
 & = & 
\tr(\begin{pmatrix}
p-2 & -1 \\ 1 & 0 
\end{pmatrix})
 \\
 & = & p -2 
\end{eqnarray*}
We remark that in case \ref{notation:p-q:item:III}, 
$\rho(\circlearrowright\beta)= 
\rho(U_{p-4})=\rho(\beta)$. 
\end{itemize}
\end{proof}

%
\section{All tubes have the same growth coefficient}
\label{sec:all-tubes}

\subsection{Growth coefficient of the two tubes $\mathcal{T}_2$ and $\mathcal{T}_3$}\label{sec:small-tubes}

In this section, we compute the quiddity sequences of the friezes corresponding to the two rank $2$ tubes arising from the triangulated twice-punctured disk. 
We refer to them as the two ``small tubes'' in this section
(note that, if $n=4$, the tube associated to the boundary of the disk is also of rank $2$, so all three tubes have the same size). 

As before, 
$T'$ is the triangulation of the annulus constructed from the triangulation $T$ of the twice-punctured disk (Section~\ref{sec:D-to-A}), 
$p$ is the number arcs in a small neighborhood of $\punctureP$ and $q$ the number of arcs in a small neighborhood of $\punctureQ$ (Notation~\ref{notation:p-q}).

Let $\gamma$ be the tagged arc connecting $\punctureP$ and $\punctureQ$ which is unnotched at either end and let $\gamma^{(\punctureP)}$, 
$\gamma^{(\punctureQ)}$ and $\gamma^{(\punctureP\punctureQ)}$ be notched at $\punctureP$, at $\punctureQ$ and at both ends, respectively (as in the first two pictures in Figure~\ref{fig:arcs-at-mouths}). 

As before, we let $a=\rho(M(\gamma'))=\rho(M(\gamma''))$. 
Recall that if $T$ contains an arc connecting $\punctureP$ and $\punctureQ$, we have $a=1$ (Remark~\ref{rem:a-is-1}).

To compute the values $\rho(M(\gamma))$, $\rho(M(\gamma^{(\punctureP)}))$, 
$\rho(M(\gamma^{(\punctureQ)}))$ and 
$\rho(M(\gamma^{(\punctureP\punctureQ)}))$ we use a result about the cluster algebra corresponding to the triangulation $T$. 

\begin{proposition}\label{prop:value-from-cl-alg}
Let $\mathcal{A}(S,T)$ be the cluster algebra associated to $(S,T)$ (\cite{FST08}). For $\mu$ a tagged arc on $(S,T)$, let $x(\mu)$ be the cluster variable associated to the arc $\mu$, and let $f(\mu)$ be the number of terms of $x(\mu)$ in the cluster variable expansion of $x(\mu)$ with respect to $T$. 
\begin{enumerate}[(a)]
    \item\label{prop:value-from-cl-alg:itm:a} 
    Then we have 
\[
f(\gamma^{(\punctureP)}) =p \cdot f(\gamma) 
\quad \mbox{and} \quad 
f(\gamma^{(\punctureP\punctureQ)}) = pq\cdot 
f(\gamma).
\]

\item\label{prop:value-from-cl-alg:itm:b} 
For $\beta\in\{\gamma,\gamma^{(\punctureP)}, \gamma^{(\punctureQ)}, \gamma^{(\punctureP\punctureQ)}\}$, we have 
\[
\rho(M(\beta)) = f(\beta)
\]
\end{enumerate}
\end{proposition}

\begin{proof}
Part \ref{prop:value-from-cl-alg:itm:a}  follows  from~\cite[Proposition 5.3]{MSW11}. 
For part 
\ref{prop:value-from-cl-alg:itm:b} we note that all these tagged arcs correspond to rigid indecomposable objects (Lemma~\ref{lm:arcs at the mouth of small tubes}). They are sent to cluster variables by the CC-map,~\cite{P08}.
\end{proof}

Now we have everything together to compute the quiddity sequences of the small tubes: 

\begin{lemma}
\label{lm:quid-small-tubes}
Let $a$, $p$, and $q$ be as defined above. 
One of the small tubes has quiddity sequence $(a,apq)$, and the other $(ap, aq)$.
\end{lemma}

\begin{proof}
By Lemma~\ref{lm:arcs at the mouth of small tubes}, the mouth of one of the rank $2$ tubes consists of the modules corresponding to tagged arcs $\gamma$ and $\gamma^{(\punctureP\punctureQ)}$, while the other consists of the modules with arcs $\gamma^{(\punctureP)}$ and $\gamma^{(\punctureQ)}$.

By Proposition~\ref{prop:value-from-cl-alg}\ref{prop:value-from-cl-alg:itm:b}, the two quiddity sequences are 
\[
(f(\gamma),f(\gamma^{(\punctureP\punctureQ)}))
\quad \mbox{and} \quad 
(f(\gamma^{(\punctureP)}),f(\gamma^{(\punctureQ)}))
\]

Noting that $f(\gamma)=a$ and using Proposition~\ref{prop:value-from-cl-alg}\ref{prop:value-from-cl-alg:itm:a}, 
these are equal to 
\[
(a,apq)
\quad \mbox{and} \quad 
(ap,aq)
\]
\end{proof}

\begin{corollary}\label{cor:quid-small-tubes}
With the above notation, the growth coefficient of each of the small tubes is 
equal to 
\[
a^2pq-2.
\]
\end{corollary}

\begin{proof}
The quiddity sequences of the two small tubes are 
$\quid=(a,apq)$ and $\quid=(ap,aq)$, respectively (Lemma~\ref{lm:quid-small-tubes}). In both cases, the corresponding frieze has entry $a^2pq-1$ in the row below. We indicate this for both friezes here: 
\[
\xymatrix@R=.6pt@C=4pt{
 1 && 1 && 1 && 1 && \ && 1 && 1 && 1 && 1 \\ 
 & a && apq && a && && && ap && aq && ap && \\
  && a^2pq-1 && a^2pq-1 && && && && a^2pq-1 && a^2pq-1 \\
 & \vdots && && \vdots   && && && \vdots && && \vdots \\
}
\]
The growth coefficient is equal to an entry in row $2$ 
minus the entry above it
in row $0$: 
\[
s(\quid)=a^2pq-1-1=a^2pq-2.
\]
\end{proof}

%
\subsection{The homogeneous tubes}~\label{sec:homogeneous}
Any homogeneous tube is of rank one, and we check that the number of submodules of its quasi-simple module is equal to $a^2pq-2$. So the growth coefficient is equal to $a^2pq-2$ 
(subtracting by $0$ because every entry in row $-1$ is $0$).

\begin{lemma}\label{lm:homogen}
The growth coefficient of any frieze arising from a homogeneous tube in affine type $D$ 
is $a^2pq-2$.    
\end{lemma}

\vskip.3cm

%
\subsection{Main result}~\label{sec:main-thm}

We are now ready to prove our main theorem. 
Let $T$ be a triangulation of a twice-punctured disk corresponding to a cluster category of affine type $D$, 
and let $p$ and $q$ be the number of arcs of $T$ incident with punctures $\punctureP$ and $\punctureQ$ respectively in  a small neighborhood (Notation~\ref{notation:p-q}). 
Let $\gamma$ denote the arc with endpoints at $\punctureP$ and at $\punctureQ$. If $\gamma$ is not in $T$, let  $M(\gamma)$ be the indecomposable module associated to $\gamma$. If $\gamma$ is in $T$, we set $M(\gamma)$ to be the zero module.

\begin{theorem}\label{theo:main}

The friezes of all tubes of a cluster category of affine type $D$ have the same growth coefficient. More precisely, if $F$ is the frieze of a tube in this category, its growth coefficient is equal to  
\[
a^2pq-2
\] where $a=\rho(M(\gamma))$.
\end{theorem}

\begin{proof}
The statement follows from our results of Sections~\ref{sec:outer-boundary},~\ref{sec:small-tubes} and~\ref{sec:homogeneous}: 
\begin{enumerate}[(i)]
\item 
Let $F$ be the infinite frieze from the tube $\mathcal{T}_1$ corresponding to the boundary of $T$, and  
let $\quid$ be its quiddity sequence. 
In Section~\ref{sec:D-to-A}, we construct a triangulated annulus $T'$ whose outer boundary corresponds to a frieze $F'$ with quiddity sequence $\quid$; since a frieze is uniquely determined by its quiddity sequence, the friezes $F'$ and $F$ are equal. 
Proposition~\ref{prop:growth-I} tells us that the growth coefficient $s(\quid)$  is $a^2pq-2$  (noting that $a=1$ if the triangulation contains an arc connecting $\punctureP$ and $\punctureQ$). 
\item 
By Corollary~\ref{cor:quid-small-tubes}, the friezes of the two rank two tubes also have this growth coefficient. 
\item 
Finally, for homogeneous tubes, the statement is given in Lemma~\ref{lm:homogen}. 
\end{enumerate}

\end{proof}

\bibliographystyle{alpha}
\bibliography{winart3} 

\newcommand{\etalchar}[1]{$^{#1}$}
\begin{thebibliography}{BMR{\etalchar{+}}06}

\bibitem[ABCP10]{ABCP10}
Ibrahim Assem, Thomas Br\"{u}stle, G.~{Charbonneau-Jodoin}, and Pierre-Guy
  Plamondon.
\newblock Gentle algebras arising from surface triangulations.
\newblock {\em Algebra Number Theory}, 4(2):201--229, 2010.

\bibitem[AD11]{AD11}
Ibrahim Assem and Gr\'{e}goire Dupont.
\newblock Friezes and a construction of the {E}uclidean cluster variables.
\newblock {\em J. Pure Appl. Algebra}, 215(10):2322--2340, 2011.

\bibitem[AP21]{AP21}
Claire Amiot and Pierre-Guy Plamondon.
\newblock The cluster category of a surface with punctures via group actions.
\newblock {\em Adv. Math.}, 389:Paper No. 107884, 63, 2021.

\bibitem[ARS10]{ARS10}
Ibrahim Assem, Christophe Reutenauer, and David Smith.
\newblock Friezes.
\newblock {\em Adv. Math.}, 225(6):3134--3165, 2010.

\bibitem[ASS06]{ASS06}
Ibrahim Assem, Daniel Simson, and Andrzej Skowro\'{n}ski.
\newblock {\em Elements of the representation theory of associative algebras.
  {V}ol. 1}, volume~65 of {\em London Mathematical Society Student Texts}.
\newblock Cambridge University Press, Cambridge, 2006.
\newblock Techniques of representation theory.

\bibitem[B{\c{C}}J{\etalchar{+}}23]{BCJKT}
Karin {Baur}, \.{I}lke {\c{C}}anak\c{c}\i, Karin~M. {Jacobsen}, Maitreyee~C.
  Kulkarni, and Gordana Todorov.
\newblock Infinite friezes and triangulations of annuli.
\newblock {\em J. Algebra Appl.}, 0(0):2450207, 2023.

\bibitem[BFPT19]{BFPT19}
Karin Baur, Klemens Fellner, Mark~J. Parsons, and Manuela Tschabold.
\newblock Growth behaviour of periodic tame friezes.
\newblock {\em Rev. Mat. Iberoam.}, 35(2):575--606, 2019.

\bibitem[BM09]{BM2009}
Karin Baur and Bethany~R. Marsh.
\newblock Frieze patterns for punctured discs.
\newblock {\em J. Algebraic Combin.}, 30(3):349--379, 2009.

\bibitem[BM12]{BM2012}
Karin Baur and Bethany~R. Marsh.
\newblock A geometric model of tube categories.
\newblock {\em Journal of Algebra}, 362:178--191, 2012.

\bibitem[BMR{\etalchar{+}}06]{BMRRT06}
Aslak~B. Buan, Bethany Marsh, Markus Reineke, Idun Reiten, and Gordana Todorov.
\newblock Tilting theory and cluster combinatorics.
\newblock {\em Adv. Math.}, 204(2):572--618, 2006.

\bibitem[BPT16]{BPT16}
Karin Baur, Mark~J. Parsons, and Manuela Tschabold.
\newblock Infinite friezes.
\newblock {\em European J. Combin.}, 54:220--237, 2016.

\bibitem[BQ15]{BQ15}
Thomas Br\"{u}stle and Yu~Qiu.
\newblock Tagged mapping class groups: {A}uslander-{R}eiten translation.
\newblock {\em Math. Z.}, 279(3-4):1103--1120, 2015.

\bibitem[BR87]{BR87}
M.~C.~R. Butler and C.~M. Ringel.
\newblock Auslander-{R}eiten sequences with few middle terms and applications
  to string algebras.
\newblock {\em Comm. Algebra}, 15(1-2):145--179, 1987.

\bibitem[BZ11]{BZ11}
Thomas Br{\"u}stle and Jie Zhang.
\newblock {On the cluster category of a marked surface without punctures}.
\newblock {\em Algebra \& Number Theory}, 5(4):529 -- 566, 2011.

\bibitem[CC73]{CC73}
John~H. Conway and Harold~S.M. Coxeter.
\newblock Triangulated polygons and frieze patterns.
\newblock {\em Math. Gaz.}, 57(400):87--94, 175--183, 1973.

\bibitem[CC06]{CC06}
Philippe Caldero and Fr\'{e}d\'{e}ric Chapoton.
\newblock Cluster algebras as {H}all algebras of quiver representations.
\newblock {\em Comment. Math. Helv.}, 81(3):595--616, 2006.

\bibitem[{Cer}11]{C11}
Giovanni {Cerulli Irelli}.
\newblock Quiver {G}rassmannians associated with string modules.
\newblock {\em J. Algebraic Combin.}, 33(2):259--276, 2011.

\bibitem[{\c{C}}FGT22]{CFGT21}
\.{I}lke {\c{C}}anak\c{c}\i, Anna Felikson, Ana {Garcia-Elsener}, and Pavel
  Tumarkin.
\newblock Friezes for a pair of pants.
\newblock {\em S\'{e}m. Lothar. Combin.}, 86B:Art. 32, 12, 2022.

\bibitem[CL12]{CL12}
Giovanni {Cerulli Irelli} and Daniel {Labardini-Fragoso}.
\newblock Quivers with potentials associated to triangulated surfaces, {P}art
  {III}: tagged triangulations and cluster monomials.
\newblock {\em Compos. Math.}, 148(6):1833--1866, 2012.

\bibitem[Cox71]{C71}
H.~S.~M. Coxeter.
\newblock Frieze patterns.
\newblock {\em Acta Arith.}, 18:297--310, 1971.

\bibitem[DR76]{DR76}
Vlastimil Dlab and Claus~Michael Ringel.
\newblock Indecomposable representations of graphs and algebras.
\newblock {\em Mem. Amer. Math. Soc.}, 6(173):v+57, 1976.

\bibitem[FST08]{FST08}
Sergey Fomin, Michael Shapiro, and Dylan Thurston.
\newblock Cluster algebras and triangulated surfaces. {I}. {C}luster complexes.
\newblock {\em Acta Math.}, 201(1):83--146, 2008.

\bibitem[FZ02]{FZ02}
Sergey Fomin and Andrei Zelevinsky.
\newblock Cluster algebras. {I}. {F}oundations.
\newblock {\em J. Amer. Math. Soc.}, 15(2):497--529, 2002.

\bibitem[GLS22]{GLS22}
Christof Gei\ss, Daniel {{L}abardini-Fragoso}, and Jan Schr\"{o}er.
\newblock Schemes of modules over gentle algebras and laminations of surfaces.
\newblock {\em Selecta Math. (N.S.)}, 28(1):Paper No. 8, 78, 2022.

\bibitem[GM22]{GM22}
Emily Gunawan and Greg Muller.
\newblock Superunitary regions of cluster algebras, 2022.
\newblock Preprint \arxiv{2208.14521v1}.

\bibitem[GMV19]{GMV19}
Emily Gunawan, Gregg Musiker, and Hannah Vogel.
\newblock Cluster algebraic interpretation of infinite friezes.
\newblock {\em European J. Combin.}, 81:22--57, 2019.

\bibitem[{Kan}22]{Kan22}
Ezgi {Kantarc{\i} {O}{\u{g}}uz}.
\newblock Oriented posets, rank matrices and q-deformed {M}arkov numbers, 2022.
\newblock Preprint \arxiv{2206.05517v3}.

\bibitem[KS11]{KS11}
Bernhard Keller and Sarah Scherotzke.
\newblock Linear recurrence relations for cluster variables of affine quivers.
\newblock {\em Adv. Math.}, 228(3):1842--1862, 2011.

\bibitem[KY24]{KY22}
Ezgi {Kantarcı Oğuz} and Emine Y{\i}ld{\i}r{\i}m.
\newblock Cluster expansions: T-walks, labeled posets and matrix calculations,
  2024.
\newblock Preprint \arxiv{2211.08011v2}.

\bibitem[{Mor}15]{Mor15}
Sophie {Morier-Genoud}.
\newblock Coxeter's frieze patterns at the crossroads of algebra, geometry and
  combinatorics.
\newblock {\em Bull. Lond. Math. Soc.}, 47(6):895--938, 2015.

\bibitem[{Mor}17]{Mor17}
Sophie {Morier-Genoud}.
\newblock Lecture notes on integrable systems and friezes, 2017.
\newblock \url{https://sophie-moriergenoud.perso.math.cnrs.fr/papers.html}.

\bibitem[MSW11]{MSW11}
Gregg Musiker, Ralf Schiffler, and Lauren Williams.
\newblock Positivity for cluster algebras from surfaces.
\newblock {\em Adv. Math.}, 227(6):2241--2308, 2011.

\bibitem[MSW13]{MSW13}
Gregg Musiker, Ralf Schiffler, and Lauren Williams.
\newblock Bases for cluster algebras from surfaces.
\newblock {\em Compos. Math.}, 149(2):217--263, 2013.

\bibitem[Pal08]{P08}
Yann Palu.
\newblock Cluster characters for 2-{C}alabi-{Y}au triangulated categories.
\newblock {\em Ann. Inst. Fourier (Grenoble)}, 58(6):2221--2248, 2008.

\bibitem[QZ17]{QZ17}
Yu~Qiu and Yu~Zhou.
\newblock Cluster categories for marked surfaces: punctured case.
\newblock {\em Compos. Math.}, 153(9):1779--1819, 2017.

\bibitem[Rie80]{R80}
Christine Riedtmann.
\newblock Algebren, {D}arstellungsk{\"o}cher, {U}eberlagerungen und zur\"uck.
\newblock {\em Comment. Math. Helvetici}, 55:199--224, 1980.

\bibitem[SS07]{SS07}
Daniel Simson and Andrzej Skowro\'{n}ski.
\newblock {\em Elements of the representation theory of associative algebras.
  {V}ol. 2}, volume~71 of {\em London Mathematical Society Student Texts}.
\newblock Cambridge University Press, Cambridge, 2007.
\newblock Tubes and concealed algebras of Euclidean type.

\bibitem[{Tsc}15]{Tsch15}
Manuela {Tschabold}.
\newblock {Arithmetic infinite friezes from punctured discs}, 2015.
\newblock Preprint \arxiv{1503.04352v3}.

\end{thebibliography}

\end{document}